\documentclass[10pt,reqno]{amsart}

\usepackage{amssymb,latexsym,mathrsfs,amsmath,amsthm}
\usepackage[colorlinks,
    linkcolor={red!50!black},
    citecolor={black!50!black},
    urlcolor={blue!80!black}]{hyperref}
\usepackage{graphicx}
\usepackage{tikz}

\newcommand{\cC}{\mathscr{C}}
\newcommand{\M}{\mathcal{M}}
\newcommand{\Z}{\mathbb{Z}}
\newcommand{\R}{\mathbb{R}}

\newcommand{\KSp}{\mathcal{X}_\mathit{Kh}}
\newcommand{\KSr}{\tilde{\mathcal{X}}_\mathit{Kh}}
\newcommand{\RP}{\mathbb{R}\mathbf{P}}

\DeclareMathOperator{\Ob}{Ob}
\DeclareMathOperator{\Sq}{Sq}
\DeclareMathOperator{\sq}{sq}
\DeclareMathOperator{\id}{id}

\newtheorem{theorem}{Theorem}[section]
\newtheorem{lemma}[theorem]{Lemma}

\newtheorem{conjecture}[theorem]{Conjecture}

\theoremstyle{definition}
\newtheorem{definition}[theorem]{Definition}

\newtheorem{remark}[theorem]{Remark}
\newtheorem{example}[theorem]{Example}

\begin{document}
\parindent0em
\setlength\parskip{.1cm}
\title[Khovanov homotopy calculations]{Khovanov homotopy calculations\\
using flow category calculus}
\author[Andrew Lobb]{Andrew Lobb}
\address{Department of Mathematical Sciences\\ Durham University, UK}
\email{andrew.lobb@durham.ac.uk}

\author[Patrick Orson]{Patrick Orson}
\address{Department of Mathematics,
Boston College, USA}
\email{patrick.orson@bc.edu}

\author[Dirk Sch\"utz]{Dirk Sch\"utz}
\address{Department of Mathematical Sciences\\ Durham University, UK}
\email{dirk.schuetz@durham.ac.uk}

\thanks{The authors were partially supported by the EPSRC Grant EP/K00591X/1. PO was supported by a CIRGET postdoctoral fellowship.}

\begin{abstract}
The Lipshitz-Sarkar stable homotopy link invariant defines Steenrod squares on the Khovanov cohomology of a link. Lipshitz-Sarkar constructed an algorithm for computing the first two Steenrod squares. We develop a new algorithm which implements the flow category simplification techniques previously defined by the authors and Dan Jones.  We give a purely combinatorial approach to calculating the second Steenrod square and Bockstein homomorphisms in Khovanov cohomology, and flow categories in general.

The new method has been implemented in a computer program by the third author and applied to large classes of knots and links. Several homotopy types not previously witnessed are observed, and more evidence is obtained that Khovanov stable homotopy types do not contain $\mathbb{C} P^2$ as a wedge summand. In fact, we are led by our calculations to formulate an even stronger conjecture in terms of $\Z/2$ summands of the cohomology.
\end{abstract}

\maketitle

\section{Introduction}

Framed flow categories were introduced by Cohen-Jones-Segal in \cite{CJS} as a way potentially to refine Floer homological invariants to space-level invariants. They were used by Lipshitz-Sarkar in \cite{LipSarKhov} to produce a stable homotopy type link invariant $\mathcal{X}_{Kh}(L)$ for links $L\subset S^3$. The cohomology of the spectrum $\mathcal{X}_{Kh}(L)$ is the Khovanov cohomology of the link $L$ and so the spectrum can be used to define Steenrod square operations on the Khovanov cohomology of the link. 
The calculation of Steenrod squares becomes important in this context as a way to distinguish interesting stable homotopy types from those which are the suspension spectrum of a wedge of Moore spaces. Moreover, it was shown in \cite{LipSarsinv} that Steenrod squares in Khovanov cohomology can be used to refine the Rasmussen~$s$-invariant of the smooth concordance class of a link, providing another important reason to be interested in their computation.

The main purpose of this paper is to describe a combinatorial algorithm for calculating Steenrod squares in the cohomology of the suspension spectrum associated to a general framed flow category based on the flow category simplifications developed in \cite{JLS2, ALPODS2}. In particular this is a new method computing the Steenrod squares in Khovanov cohomology.

This combinatorial algorithm has been turned into a computer program by the third author. We present some of the more significant computations, including several stable homotopy type wedge summands not previously witnessed, and more evidence for the possibility that $\mathbb{C} P^2$ never appears as a wedge summand of the Khovanov stable homotopy type of any link (see \cite[Question 5.2]{LipSarSq}). Guided by the calculations, we extend this idea to a conjecture about 2-torsion in Chang space wedge summands (see Conjecture \ref{conj:chang}).

\subsection{1-flow categories and the second Steenrod square}

In \cite{LipSarSq}, Lipshitz-Sarkar already describe an algorithm for computing Steenrod squares from a framed flow category, but it is valid only in the specific context of flow categories arising directly from their construction in \cite{LipSarKhov} applied to a link $L\subset S^3$. Our algorithm is quite different from theirs, it works for any framed flow category, and indeed in the more general setting of what we call a \emph{framed 1-flow category}. Any framed flow category restricts to one of these framed 1-flow categories, but a framed 1-flow category is a simpler object. 
In some sense it is nothing but a cochain complex with the minimal amount of extra data added so that one can define a combinatorial second Steenrod square on the cohomology (see Subsection \ref{subsec:1flow} for precise details of this statement).

Put in terms of knot theory and the Lipshitz-Sarkar construction \cite{LipSarKhov}, the Khovanov cochain complex of a link can only `see' the information of the vertices and edges of the Khovanov cube of the link.  The complete Lipshitz-Sarkar framed flow category can see all subcubes of the Khovanov cube of the link, and how they interrelate.  The 1-flow category of a link is a collection of data which can see the information of how 2-dimensional faces of the Khovanov cube relate edges and vertices.  However, this amount of information is sufficient if all you want is to compute the second Steenrod square.

\subsection{Combinatorial arguments}
While a link in $S^3$ is a geometric object, the input for Khovanov homology and for the Khovanov homotopy is the Khovanov cube, which can be viewed as a combinatorial object. We believe that there is some benefit to giving proofs and constructions purely combinatorially, particularly when a combinatorial input is involved. It is also possible that the combinatorics are amenable to generalizations and extensions in ways that non-combinatorial arguments may not be.

In the particular cases of the constructions considered in this paper (which arise originally from framings of embedded arcs and circles), it is a considerable effort to turn the topological data of a framed flow category into combinatorial data sufficient to make computations of the second Steenrod square.  Working combinatorially from the start, this effort is sidestepped, leading to more elementary proofs.

\subsection{Plan of the paper}
We begin the paper by presenting in Section \ref{sec:calculations} some calculations made by the program implementing the algorithm due to the third author \cite{SchuetzKJ}.  We discuss what these calculations suggest for the conjectural picture of stable homotopy types realized by the Lipshitz-Sarkar construction.

For knots and links with more than 16 crossings the associated flow categories returned by the Lipshitz-Sarkar construction become very large, and we found that the new algorithm, in particular when combined with the results of \cite{JLS} on special diagrams, greatly speeds up the Steenrod square computations in these cases.

Section \ref{sec:comb_steen_square} is dedicated to a combinatorial construction and proof of the following theorem, stated here in abbreviated form:

\medskip

{\bf Theorem \ref{thm:steensquare}.}$\,\,$ {\sl Let $(\cC, s, f)$ be a framed 1-flow category. Then there is a well-defined linear map $\Sq^2\colon H^*(\cC;\Z/2) \to H^{*+2}(\cC;\Z/2)$.
}

\medskip

What we mean by `well-defined' in this context we postpone until later to make precise.  Of course, the map defined in this theorem should (and does) agree with the second Steenrod square when computed in the context of the Khovanov homotopy type and its relatives, as we discuss at the end of Section \ref{sec:comb_steen_square}.

In Sections \ref{sec:handle} and \ref{sec:snf_for_flow_categories} we turn to the algorithm.  The algorithm makes heavy use of the \emph{flow category moves} in \cite{JLS2} and \cite{ALPODS2}, or rather their 1-flow category restrictions.  The moves -- called \emph{handle slides}, \emph{handle cancellation}, and \emph{Whitney trick} -- for modifying a framed flow category, are based on ideas from Morse theory.  The moves do not change the stable homotopy type associated to the flow category, but we show in \cite[\textsection 6]{ALPODS2} that they can be used to simplify the flow category itself in an algorithmic way.  To take advantage of this idea for our eventual Steenrod square algorithm, we describe explicitly in Section \ref{sec:handle} how our flow category moves restrict to framed 1-flow categories.

In Section \ref{sec:snf_for_flow_categories} we give a new combinatorial way to calculate the second Steenrod square.
Subsection \ref{subsec:partial} gives a refinement of the auxiliary data used to compute the second Steenrod square, and Subsection \ref{subsec:steen_sq_and_partial_comb_matchings} shows that this refinement suffices.  Subsection \ref{subsec:effects} gives a treatment of the flow category moves in the contexts necessary for the algorithm, while Subsection \ref{subsec:algorithm} gives a complete description of the algorithm.  The algorithm works by first simplifying the 1-flow category using the flow category moves, then applying the new combinatorial Steenrod square process to obtain the result.

\section{Calculations}\label{sec:calculations}

A computer program implementing the described algorithm has been developed by the third author, see \cite{SchuetzKJ}. Both an integral and a mod $4$ version have been written.  The mod $4$ version is noticeably faster when applied to diagrams with a larger number of crossings.

For more complicated diagrams such as for the torus knot $T_{7,4}$ the new program is faster than a previously written program of the third author. The new program is capable of doing calculations on the diagram for the torus knot $T_{6,5}$ given in Figure \ref{fig:tori}, although for different quantum degrees the time consumption varies greatly. As the algorithm still deals with a global approach to the diagram, the program has large memory requirements.

Note that the algorithm starts with objects of top degree in the normalization process. This can lead to a quite different performance depending on whether a knot or its mirror is considered!  One may wonder whether a more balanced algorithm may lead to better results.

\begin{figure}[ht]
\begin{center}
\includegraphics[height= 5cm, width = 5cm]{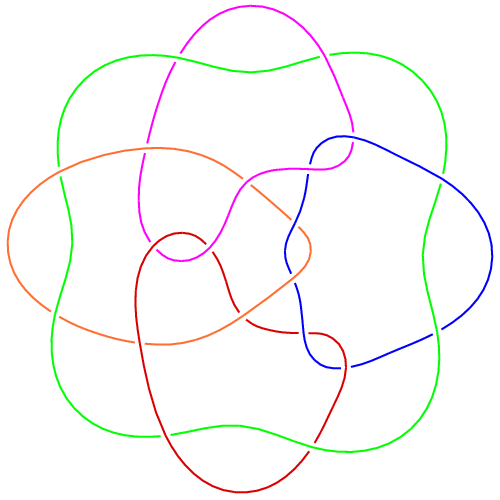}
\includegraphics[height= 5cm, width = 5cm]{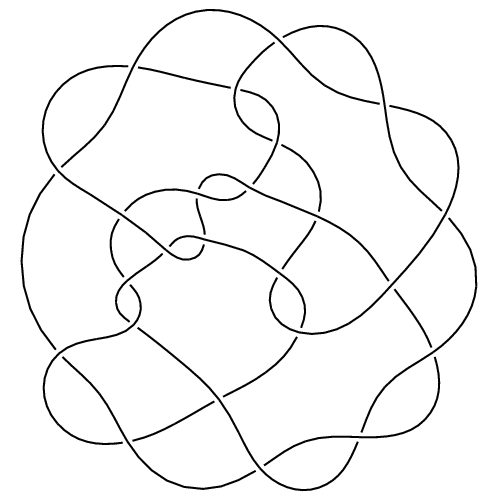}
\caption{Glued diagrams for $T_{5,5}$ (left) and $T_{6,5}$ (right).  For the definition of \emph{glued diagram} see \cite{JLS}.}
\label{fig:tori}
\end{center}
\end{figure}

We denote by $\KSp^q(L)$ the stable homotopy type of the link $L$ in quantum degree $q$, and we write $\KSr^q(L)$ for the reduced stable homotopy type, compare \cite[\S 8]{LipSarKhov}.

\subsection{The torus link $T_{5,5}$}

We used the diagram of Figure \ref{fig:tori} to determine the Lipshitz-Sarkar stable homotopy types of $T_{5,5}$ for various quantum degrees. Figure \ref{fig:KhovanovT55} lists the Khovanov cohomology of $T_{5,5}$ which potentially contain non-trivial second Steenrod squares. For $q=21$ we get the Chang\footnote{See \cite{baues} for the definition of Chang space, and also for the definition of the various Baues-Hennes spaces.} space $X(_2\eta)$ which is already present in the stable homotopy type of $T_{3,3}$ and predicted by the stability result of Willis \cite{Willis}. Similarly, for $q=23$ we get no non-trivial Steenrod squares, in line with the result we get for $T_{5,4}$ for $q=19$.

\begin{figure}[ht]
\begin{center}
\includegraphics[width=7cm]{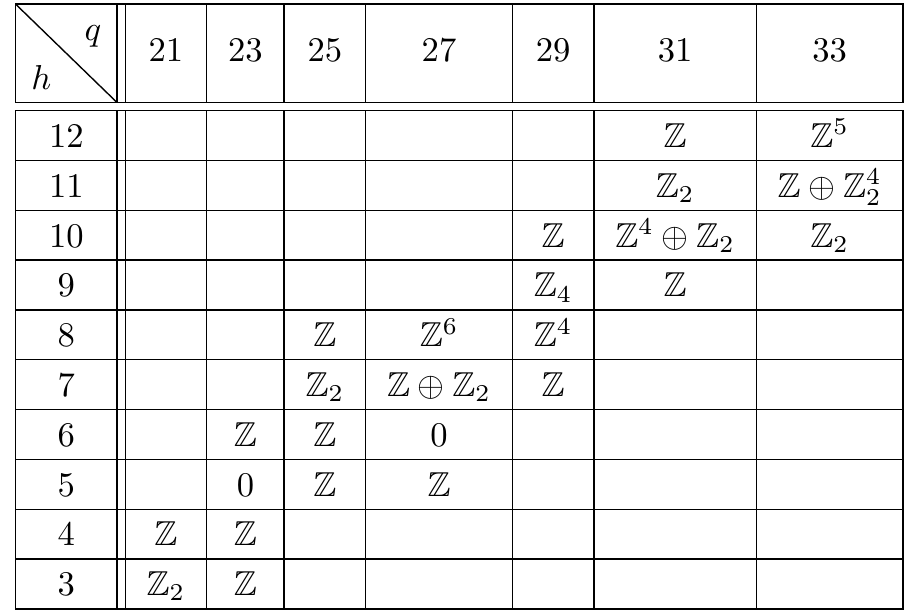}
\caption{The Khovanov cohomology of $T_{5,5}$ in quantum degrees from $q=21$ to $33$.}
\label{fig:KhovanovT55}
\end{center}
\end{figure}

Both of these calculations are done rather fast. For $q=25$ the calculation is still fairly efficient, but one can get the result even quicker by considering the reduced stable homotopy type.

The reduced Khovanov cohomology of $T_{5,5}$ is given in Figure \ref{fig:redKhovT55}, and calculations show that $\KSr^{26}(T_{5,5})$ is the central Baues-Hennes space $X(\eta2\xi)$, while $\KSr^{32}(T_{5,5})$ is the wedge of a Chang space $X(_2\eta)$ with Moore spaces. The central Baues-Hennes space $X(\eta2\xi)$ can be realized as an appropriate suspension of $\RP^5/\RP^1$.

\begin{figure}[ht]
\begin{center}
\includegraphics[width=5cm]{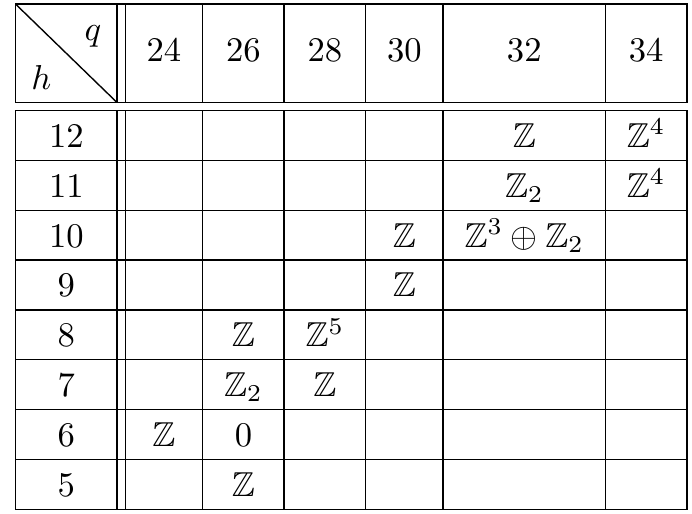}
\caption{The reduced Khovanov cohomology of $T_{5,5}$ in quantum degrees from $q=24$ to $34$.}
\label{fig:redKhovT55}
\end{center}
\end{figure}

There is a cofibration sequence $\KSr^{q-1}(T_{5,5}) \hookrightarrow \KSp^q(T_{5,5}) \twoheadrightarrow \KSr^{q+1}(T_{5,5})$ by \cite[Thm.8.1]{LipSarKhov}. From the resulting long exact sequence in cohomology and the naturality of Steenrod squares we get that both $\KSp^{25}(T_{5,5})$ and $\KSp^{27}(T_{5,5})$ contain $X(\eta2\xi)$ and the rest are Moore spaces.

For $q\geq 29$ the calculations become very time and memory consuming, and we have not finished them. Particularly $q=29$ would be interesting, as there is no $\Z/2$ summand in the Khovanov cohomology. One may expect a trivial second Steenrod square from degree $7$ to $9$ because of naturality with $T_{5,4}$, but a non-trivial second Steenrod square from degree $8$ to $10$ would lead to a previously unobserved Chang space.

\subsection{The torus knot $T_{6,5}$}

Similarly we can do calculations for $T_{6,5}$ using the right diagram of Figure \ref{fig:tori}. For $q=29$ we also get a central Baues-Hennes space $X(\eta2\xi)$, which should come from the $T_{5,5}$ result, although it does not fall into the stable range described by Willis \cite{Willis}. Again this Baues-Hennes space appears in the reduced stable homotopy type for $q=30$, which means it also appears in the unreduced case for $q=31$.

\subsection{Knots with non-trivial $\Sq^3$}

The disjoint union of three trefoils is known to contain the cyclic Baues-Hennes space $X(\xi^2\eta_2,\id)$, see \cite{LawLipSar} and \cite{ALPODS2}, giving the simplest example of a link with non-trivial $\Sq^3$. Recall that $\Sq^1\Sq^2 = \Sq^3$ by an Adem relation. Going through the prime knot tables of \cite{knotscape}, we found exactly one knot with at most 13 crossings that admits a non-trivial $\Sq^3$, namely $13n3663$ in quantum degree $q=1$. The resulting Baues-Hennes space is again $X(\xi^2\eta_2,\id)$. Figure \ref{fig_squarethree} lists the prime knots with 14 crossings that admit this Baues-Hennes space. No other Baues-Hennes spaces have been found for prime knots with up to 14 crossings, although there could be Baues-Hennes spaces corresponding to $\varepsilon$-words among them.
\begin{figure}[ht]
\begin{tabular}{|c|c||c|c|}
\hline
Knot & $q$ & Knot & $q$ \\
\hline
\hline
14n10164 & $-3$ & 14n21627 & $-9$\\
14n10510 & $-7$ & 14n22180 & $3$\\
14n18918 & $-7$ & 14n22185 & $-1$\\
14n18935 & $15$ & 14n22589 & $-1$\\
14n19177 & $-1$ & 14n23524 & $-3$\\
14n19265 & $-3$ & 14n26039 & $-1$\\
14n19315 & $9$ & 14n26580 & $9$\\
\hline
\end{tabular}
\caption{Prime knots with 14 crossings containing a cyclic Baues-Hennes space $X(\xi^2\eta_2,\id)$.}
\label{fig_squarethree}
\end{figure}

\subsection{Knots with Chang spaces involving $4$-torsion}

There are not that many prime knots with at most 16 crossings that have $4$-torsion in their Khovanov cohomology. We have tested these knots for whether their stable homotopy type contains a Chang space involving $4$-torsion. Those which do are listed in Figure \ref{fig_chang4tor}.
\begin{figure}[ht]
\includegraphics[width = 12.6cm]{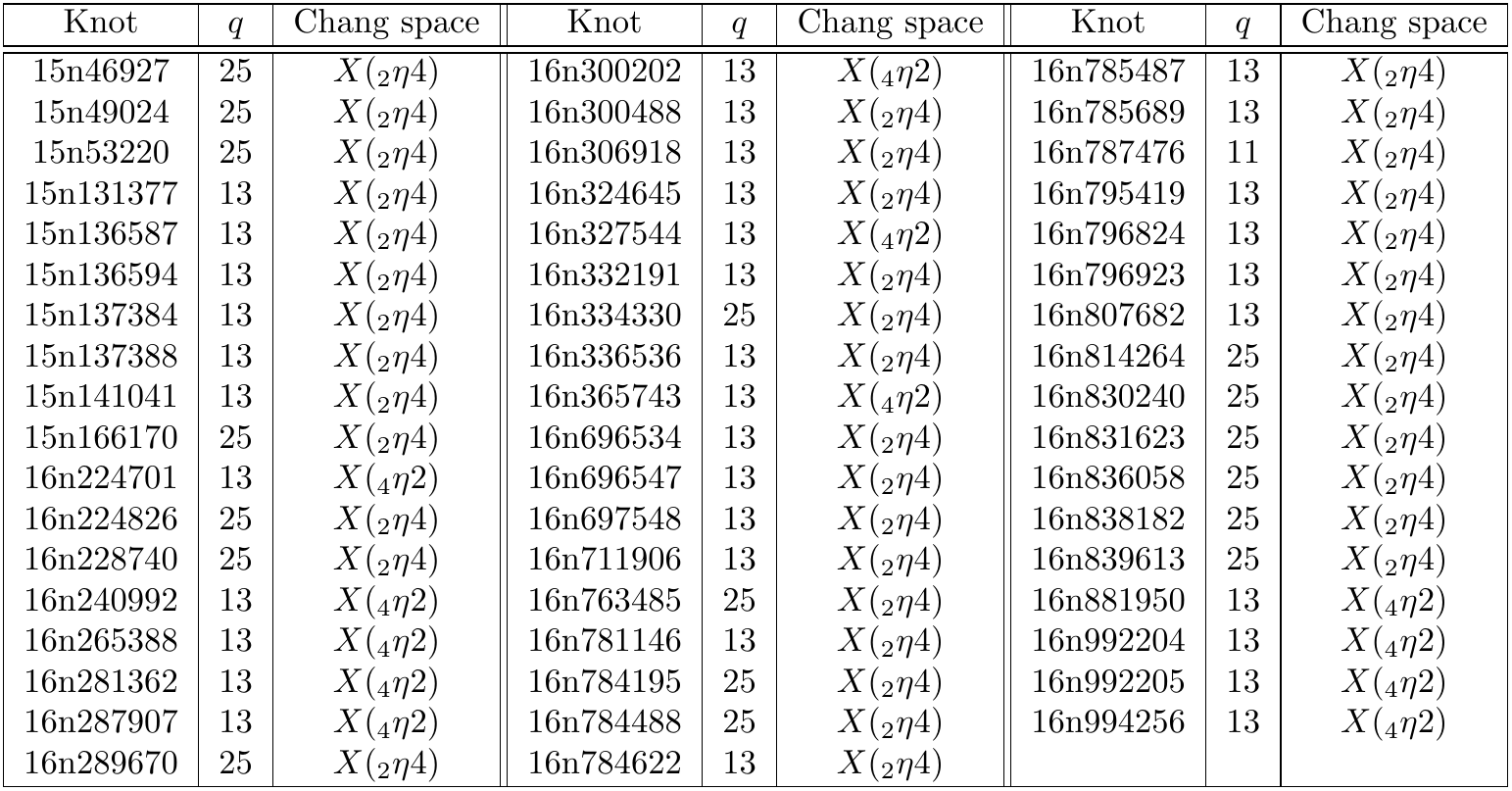}
\caption{Prime knots with up to 16 crossings containing a Chang space involving $4$-torsion.}
\label{fig_chang4tor}
\end{figure}
The knots have been mirrored so that the $4$-torsion appears in a positive quantum degree. Note that all the occurring Chang spaces also contain $2$-torsion, and the Chang space appears in quantum degree $13$ or $25$ (with one exception, where $q=11$). The table only lists the summand in the stable homotopy type which contains the $4$-torsion. None of the knots with $4$-torsion in their Khovanov cohomology has more than one $\Z/4$ summand in their cohomology. On the basis of these calculations, we make the following tentative conjecture.

\begin{conjecture}\label{conj:chang}If a Chang space occurs as a wedge summand of the Khovanov stable
homotopy type of a link, then that Chang space has a $\Z/2$ summand in its cohomology.
\end{conjecture}

Note that, as $\mathbb{C} P^2$ has the stable homotopy type of the Chang space $X(\eta)$, this conjecture includes the statement that~$\mathbb{C} P^2$ does not appear as a wedge summand of the Khovanov stable homotopy type of a link.

\section{Combinatorial Steenrod squares}
\label{sec:comb_steen_square}
This section is in large part dedicated to the statement and combinatorial proof of Theorem \ref{thm:steensquare}.  This theorem will not be surprising, given its topological motivations.  The price paid for keeping the arguments in this section combinatorial is a somewhat lengthy although elementary proof.

We start in Subsection \ref{subsec:1flow} giving the definition of our principal object of interest: a 1-flow category.  Subsection \ref{subsec:secgraphstr} formalizes the auxiliary information needed to make a computation of the second Steenrod square, and does it in such a way as to be useful to us in giving the algorithm later.  Finally Subsection \ref{subsec:moves} gives the proof of Theorem \ref{thm:steensquare}.  We conclude with a few remarks concerning topological arguments.

\subsection{1-flow categories}\label{subsec:1flow}

In the combinatorial algorithm of \cite{LipSarSq} for calculating the Steenrod squares on Khovanov cohomology, it was already observed that only a very restricted amount of the data of the flow category is actually required for the computation. In this paper we formalise this restricted amount of information in the following object.

\begin{definition}\label{def:1flow}
 A \em $1$-flow category $\cC$ \em consists of a finite set $\Ob(\cC)$, a function $|\cdot|\colon \Ob(\cC)\to \Z$ called the \em grading\em, and for each pair $a,b\in \Ob(\cC)$ with $|a|-|b|=1$ or $2$ a \em moduli space \em $\M(a,b)$ which is a compact manifold of dimension $|a|-|b|-1$, satisfying the following

\begin{itemize}
 \item {\bf Boundary Condition:}
If $a,c\in \Ob(\cC)$ with $|a|-|c|=2$, then the boundary of $\M(a,b)$ is given by
\[
 \partial \M(a,c)=\coprod_{b\in \Ob(\cC),|b|=|a|-1} \M(b,c)\times \M(a,b)
\]
\end{itemize}

If $|a|-|d|=3$ we do not need a moduli space $\M(a,d)$. However, we define the boundary of $\M(a,d)$ as
\begin{align*}
 \partial \M(a,d)= & \coprod_{b\in \Ob(\cC),|b|=|a|-1} \M(b,d)\times \M(a,b) \\
 & \cup \coprod_{c\in \Ob(\cC),|c|=|a|-2}\M(c,d)\times \M(a,c)
\end{align*}
Notice that the two disjoint unions have a common subset, which is
\[
 \coprod_{\stackrel{(b,c)\in \Ob(\cC)\times \Ob(\cC)}{|b|=|c|+1=|d|+2}} \M(c,d)\times \M(b,c)\times \M(a,b)
\]
It is easy to see that $\partial\M(a,d)$ is a disjoint union of components which are homeomorphic to circles.

\end{definition}

We caution that a $1$-flow category is not a flow category. But one can usually think of a $1$-flow category as a flow category where the higher dimensional moduli spaces are ignored (although not always, as we show in Example \ref{ex:adem}).

\begin{definition}
 Let $\cC$ be a $1$-flow category. A \em sign assignment $s$ \em for $\cC$ is an assignment $s(P)\in \{0,1\}$ for every point $P$ in a $0$-dimensional moduli space $\M(a,b)$ with the property that if $(P_1,Q_1)\in \M(b_1,c)\times \M(a,b_1)$ and $(P_2,Q_2)\in \M(b_2,c)\times \M(a,b_2)$ are the boundary of an interval component in $\M(a,c)$, then 
\[
 s(P_1)+s(Q_1)+s(P_2)+s(Q_2)=1.
\]
A \em pre-framing $f$ \em of $\cC$ is an assignment $f(C)\in \{0,1\}$ for every component $C\subset \M(a,c)$ of a $1$-dimensional moduli space.
\end{definition}

The chain complex $C_\ast(\cC)$ is defined to have modules freely generated over $\Z$ by the objects of $\cC$, which are homologically graded by the grading function of $\cC$. A sign assignment allows us to define a differential on this graded module in the obvious way. The purpose of a pre-framing is to give us enough information to define a cohomology operation combinatorially on the resulting cohomology with $\Z/2$-coefficients. However, an arbitrary choice of pre-framing will not be sufficient for this purpose, and there must be a certain consistency in our choices as we now discuss.

Note that the connected components $C$ of a $1$-dimensional moduli space are either intervals or circles. A circle embedded in $\R^n$ with $n\geq 4$ can be framed in two ways, one is trivial in the sense that it extends to a framing of the disc, and there is a non-trivial framing. We shall also refer to the trivial framing as the \emph{standard framing}. Similarly, if an interval is embedded in $[0,\infty)\times \R^n$, with boundary embedded in $\{0\}\times \R^n$, and satisfying a transversality condition near the boundary, then there are two ways of framing it (provided a framing on the boundary is fixed and that this boundary framing is different at the two endpoints). 
Calling one of these two framings `standard' or `non-standard' requires consistent choices. Such choices were made in \cite[\textsection 3]{LipSarSq}. We will not repeat that somewhat technical discussion here, as we shall only need the combinatorial consequences.

When $f(C)=0$ we will call this a trivial framing for $C$, and when $f(C)=1$ we will call this a non-trivial framing. In order for this language to be meaningful, we need to add a consistency condition in terms of an induced framing on $\partial\M(a,d)$.

Let $C$ be a component in $\partial \M(a,d)$. If $C$ is a circle, which is either of the form $\{P\}\times S^1\subset \M(c,d)\times \M(a,c)$ or $S^1\times \{Q\}\subset \M(b,d)\times \M(a,b)$, we define 
\[
\tilde{f}(C)=1+f(S^1)\in \Z/2,
\]
 where $f(S^1)$ is the framing value of the circle. If $C$ is a union of intervals $\{P\}\times J\subset \M(c,d)\times \M(a,c)$ and $I \times \{Q\}\subset \M(b,d)\times \M(a,b)$, we define
\[
 \tilde{f}(C) = \sum_{I \times \{Q\}}f(I)+ \sum_{\{P\}\times J} (1+s(P)+f(J))\in \Z/2
\]
where the first sum is over all intervals of the form $I\times\{Q\}\subset C$ and the second sum over all intervals of the form $\{P\}\times J$.

\begin{definition}
 Let $\cC$ be a $1$-flow category, $s$ a sign assignment and $f$ a pre-framing of $\cC$. Then $(\cC,s,f)$ is called a \em framed $1$-flow category\em, if the following \em compatibility condition \em is satisfied for $f$.
\begin{itemize}
 \item Let $a,d\in \Ob(\cC)$ satisfy $|a|=|d|+3$. Then
\[
 \sum_{C} (1+\tilde{f}(C)) = 0 \in \Z/2
\]
where the sum is taken over all components in $\partial \M(a,d)$.
\end{itemize}
If $f$ satisfies the condition with respect to $s$, we call $f$ a \em framing of $\cC$\em.
\end{definition}

\begin{definition}
Let $(\cC,s,f)$ be a framed $1$-flow category. For objects $a,b\in \Ob(\cC)$ with $|a|=|b|+1$ we define $[a:b]\in \Z$ as
\[
 [a:b] = \sum_{A\in \M(a,b)} (-1)^{s(A)}.
\]
\end{definition}

\begin{remark}
A framed flow category in the sense of \cite{LipSarKhov} determines a framed $1$-flow category. The framing values for each interval component can be derived using the description in \cite{LipSarSq}, see also \cite{JLS}. We omit the slightly cumbersome derivation of this, but the reader is invited to check that the specific framing values given in \cite{LipSarSq} and \cite{JLS} do indeed satisfy the compatibility condition.
\end{remark}

\subsection{Special graph structures}\label{subsec:secgraphstr}

\begin{definition}
Let $V$ be a finite set. An \em edge \em is a subset $e\subset V$ with two elements. A \em directed edge \em is an edge $e$ together with a function $d\colon e \to V$ satisfying $d(e)\subset e$.
\end{definition}

We can think of a directed edge $e$ pointing towards $d(e)$, or write it as a pair $e=(v,w)$ with $w=d(e)$.

\begin{definition}\label{def:graphstructure}
 A \em special graph structure \em $\Gamma=\Gamma(V,E,E',E'',L,s,f)$ consists of a set of vertices $V$, a function $s\colon V \to \{0,1\}$, a set of edges $E$, a subset $E'\subset E$ together with a function $f\colon E'\to \{0,1\}$, a subset $E'' \subset E-E'$ of directed edges, and a set of loops $L$. All sets are finite. It has to satisfy the following conditions:
\begin{enumerate}
 \item \label{item:gs1}Each vertex is contained in at least one edge, and at most in two edges. We denote the set of vertices contained in only one edge by $\partial V$ and call these the boundary points.
 \item \label{item:gs2}If $v\in \partial V$, then the unique edge $e(v)$ with $v\in e(v)$ satisfies $e(v)\in E'$.
 \item \label{item:gs3}If $v \in V-\partial V$, then there is an edge $e_1\in E'$ with $v\in e_1$, and an edge $e_2\in E-E'$ with $v\in e_2$.
 \item \label{item:gs4}If $e\in E'$ and $e=\{v_1,v_2\}$, then $s(v_1)\not=s(v_2)$.
 \item \label{item:gs5}If $e\in E-E'$ and $e=\{v_1,v_2\}$, then $s(v_1)=s(v_2)$ if and only if $e\in E''$.
\end{enumerate}
The edges $E$ determine an equivalence relation on the vertex set $V$ in the obvious way. The set of equivalence classes under this relation, taken together with the set $L$ of loops, forms the set of \em components \em of $\Gamma$. Notice that $L$ is a set independent from possible loops made from edges in $E$.
\end{definition}

\begin{example}\label{exam:first}
Let ($\cC,s,f)$ be a framed $1$-flow category and $a,c\in \Ob(\cC)$ with $|a|-|c|=2$. Define a special graph structure $\Gamma(a,c)$ as follows. The vertex set $V$ consists of all points $(B,A)\in \M(b,c)\times \M(a,b)$ for some $b\in \Ob(\cC)$ with $s(B,A)=s(B)+s(A)$. The edges $E'=E$ are determined by the intervals in $\M(a,c)$. The set of loops $L$ in $\Gamma(a,c)$ is given by the set of non-trivially framed circles in $\M(a,c)$. The functions $s$ and $f$ in $\Gamma(a,c)$ are determined by the sign assignment and framing of $(\cC,s,f)$.
\end{example}

\begin{definition}
Let $\Gamma=\Gamma(V,E,E',E'',L,s,f)$ be a special graph structure. For a component $C$ of $\Gamma$ which contains a vertex, let $F(C)$ be the sum of the framing values $f(e')$ where $e'\in E'$ is in $C$. Also, let $D(C)\in \Z/2$ be the number of oriented edges in $C$ which point in a chosen given direction.
\end{definition}

\begin{lemma}
 In the definition of $D(C)$, the value does not depend on the choice of given direction.
\end{lemma}

\begin{proof}
Let $C$ be a component and $v$ a vertex in $C$. We need to show that the number of directed edges in $C$ is even. First note that the total number of edges in $C$ is even by Definition \ref{def:graphstructure} (\ref{item:gs3}). Let $F''$ be the directed edges in $C$ and let $F$ be the non-directed edges in $C$. Note that non-directed edges change the sign of their endpoints by Definition \ref{def:graphstructure} (\ref{item:gs4}) and (\ref{item:gs5}), while directed edges do not, by Definition \ref{def:graphstructure} (\ref{item:gs5}). Since $C$ is a circle, we have to have an even number of edges in $F$. It follows that the number of directed edges is also even.
\end{proof}

Let $(\cC,s,f)$ be a framed 1-flow category. A cochain $\varphi\in C^k(\cC;\Z/2)$ can be identified with a function $\varphi\colon \Ob_k(\cC)\to \Z/2$, where $\Ob_k(\cC)\subset\Ob(\cC)$ consists of those objects $c\in \Ob(\cC)$ with $|c|=k$. For each cochain $\varphi$ there exists a unique set $\{c_1,\ldots,c_l\}\subset \Ob(\cC)$ with $\varphi(c_i)=1$ for all $i=1,\ldots,l$, and $\varphi(c)=0$ if $c\not=c_i$ for all $i=1,\ldots,l$. We say that $\varphi$ is \em represented \em by $c_1,\ldots,c_l$. If $\varphi$ is represented by exactly one $c$, we will simply write $c=\varphi$ by an abuse of notation.

Now suppose $\varphi\in C^k(\cC;\Z/2)$ is a cocycle represented by $c_1,\ldots,c_l\in \Ob(\cC)$. Let $b_1,\ldots,b_m\in \Ob(\cC)$ be all objects that have a non-empty $0$-dimensional moduli space $\M(b_i,c_j)$ for some $j\in\{1,\ldots,l\}$. 

If we fix $b_i$, the union
\[
 \M(b_i,\varphi) := \bigcup_{j=1}^l \M(b_i,c_j)
\]
is a set of even cardinality because $\varphi$ is a cocycle with $\Z/2$-coefficients. 

\begin{definition}Let $(\cC,s,f)$ be a framed 1-flow category and $\varphi\in C^k(\cC;\Z/2)$ be a cocycle. Writing $b_1,\dots b_m$ as above we may choose a partition of the elements of $\M(b_i,\varphi)$ into ordered pairs. If we make this choice for each $i=1,\ldots,m$, the overall choice is called a \em combinatorial boundary matching \em $\mathcal{C}$ for $\varphi$.
\end{definition}

Now suppose we have made a choice of combinatorial boundary matching $\mathcal{C}$ for a cocycle $\varphi\in C^k(\cC;\Z/2)$ and let $a\in \Ob(\cC)$ satisfy $|a|=k+2$. Define a special graph structure $\Gamma_\mathcal{C}(a,\varphi)$ as follows. We begin with the disjoint union of the special graph structures $\Gamma(a,c_j)$, defined as in Example \ref{exam:first}. Note that if $\M(a,c_i)=\emptyset$, we can also use the empty set for $\Gamma(a,c_j)$. Let $(B_1,B_2)$ be a pair in $\mathcal{C}$. We then add an edge between $(B_1,A)$ and $(B_2,A)$ for all $A\in \M(a,b_i)$ with $B_k\in \M(b_i,c_{j_k})$. If $s(B_1)=s(B_2)$, this edge is directed towards $B_2$, otherwise it is undirected. This defines $\Gamma_\mathcal{C}(a,\varphi)$.

Notice that by construction all vertices of $\Gamma_\mathcal{C}(a,\varphi)$ have valency two, so that the set of components $C$ of $\Gamma_\mathcal{C}(a,\varphi)$ consists of a disjoint union of circles together with the loop set $L$.

\begin{definition}\label{def:sqca}Let $(\cC,s,f)$ be a framed $1$-flow category and $\varphi\in C^k(\cC;\Z/2)$ a cocycle. Then define a cochain $\sq^\varphi\colon C_{k+2}(\cC;\Z/2)\to \Z/2$ by 
\[
\sq^\varphi(a):=|L|+\sum_{C}\left(1+F(C)+D(C)\right)\in\Z/2,
\]
where $L$ is the loop set of $\Gamma_\mathcal{C}(a,\varphi)$ and the sum is taken over all components $C$ of~$\Gamma_\mathcal{C}(a,\varphi)$ containing a vertex.
\end{definition}

The following theorem states that we now have enough combinatorial information to define a Steenrod square for 1-flow categories.

\begin{theorem}\label{thm:steensquare}
Let $(\cC,s,f)$ be a framed $1$-flow category and $\varphi\in C^k(\cC;\Z/2)$ a cocycle. Then $\sq^\varphi$ is a cocycle and the cohomology class of $\sq^\varphi$ does not depend on the choice of combinatorial boundary matching for $\varphi$. Furthermore, there is a well defined linear map \[\Sq^2\colon H^k(\cC;\Z/2) \to H^{k+2}(\cC;\Z/2);\qquad z\mapsto [\sq^\varphi],\]where $\varphi$ is any choice of cocycle representing the cohomology class $z$.
\end{theorem}

We will prove the theorem at the end of Subsection \ref{subsec:moves}, once we have developed the necessary combinatorial machinery to do so.

\begin{remark}
The definition of $\sq$ is of course motivated by the formula given in \cite{LipSarSq}, and if the framed $1$-flow category is induced from a framed flow category, the resulting linear map agrees with the second Steenrod square operation, compare also \cite{JLS}.

\end{remark}

\begin{example}\label{example:1}
Consider the $1$-flow category $\cC$ consisting of four objects $a,b_1,b_2,c$ and $0$-dimensional moduli spaces 
\begin{align*}
 \M(a,b_1)&=\{X_1,X_2\}, & \M(a,b_2)&=\{Y_1,Y_2\}\\
 \M(b_1,c)&=\{Z_1,Z_2\}, \mbox{ and} & \M(b_2,c)&=\{W_1,W_2\}.
\end{align*}
Framing $Y_1$ and $Y_2$ negatively, while all other points are framed positively ensures that we can extend $\cC$ to a framed $1$-flow category having four standardly framed intervals between the eight endpoints $(Z_i,X_j)$ and $(W_k,Y_l)$ for $i,j,k,l = 1,2$.

Two possible choices are given by having intervals between $(Z_i,X_j)$ and $(W_i,Y_j)$, or between $(Z_i,X_j)$ and $(W_j,Y_i)$.

If we use the combinatorial matching $(Z_1,Z_2)$ and $(W_1,W_2)$ (note that our only choice is the order for both pairs), the resulting circles occur as follows.
\[
\begin{tikzpicture}
\draw[|-|] (-4,-1) -- (-2,-1);
\draw[|-|] (-4,0) -- (-2,0);
\draw[|-|] (-4,1) -- (-2,1);
\draw[|-|] (-4,2) -- (-2,2);
\draw[|-|] (0,-1) -- (2,-1);
\draw[|-|] (0,0) -- (2,0);
\draw[|-|] (0,1) -- (2,1);
\draw[|-|] (0,2) -- (2,2);
\draw (-4,1) edge[bend right = 30, blue, ->] (-4,-1);
\draw (-4,2) edge[bend right = 30, blue, ->] (-4,0);
\draw (-2,1) edge[bend left = 30, blue, ->] (-2,-1);
\draw (-2,2) edge[bend left = 30, blue, ->] (-2,0);
\draw (0,1) edge[bend right = 30, blue, ->] (0,-1);
\draw (0,2) edge[bend right = 30, blue, ->] (0,0);
\draw (2,0) edge[bend left = 30, blue, ->] (2,-1);
\draw (2,2) edge[bend left = 30, blue, ->] (2,1);
\draw[] (-3.7,-1) node[above,scale=0.6]{$Z_2,X_2$};
\draw[] (-2.3,-1) node[above,scale=0.6]{$W_2,Y_2$};
\draw[] (-3.7,0) node[above,scale=0.6]{$Z_2,X_1$};
\draw[] (-2.3,0) node[above,scale=0.6]{$W_2,Y_1$};
\draw[] (-3.7,1) node[above,scale=0.6]{$Z_1,X_2$};
\draw[] (-2.3,1) node[above,scale=0.6]{$W_1,Y_2$};
\draw[] (-3.7,2) node[above,scale=0.6]{$Z_1,X_1$};
\draw[] (-2.3,2) node[above,scale=0.6]{$W_1,Y_1$};
\draw[] (0.3,-1) node[above,scale=0.6]{$Z_2,X_2$};
\draw[] (1.7,-1) node[above,scale=0.6]{$W_2,Y_2$};
\draw[] (0.3,0) node[above,scale=0.6]{$Z_2,X_1$};
\draw[] (1.7,0) node[above,scale=0.6]{$W_1,Y_2$};
\draw[] (0.3,1) node[above,scale=0.6]{$Z_1,X_2$};
\draw[] (1.7,1) node[above,scale=0.6]{$W_2,Y_1$};
\draw[] (0.3,2) node[above,scale=0.6]{$Z_1,X_1$};
\draw[] (1.7,2) node[above,scale=0.6]{$W_1,Y_1$};
\end{tikzpicture}
\]
The left side is what we get for connecting the points $(Z_i,X_j)$ with $(W_i,Y_j)$, and the right side when we connect $(Z_i,X_j)$ with $(W_j,Y_i)$. On the left we get for both circles $D(C) = 1$, so $\sq^c(a)=0$ which results in the trivial Steenrod square. On the right there is only one circle, with $D(C)=0$, so $\sq^c(a)=1$, resulting in a non-trivial Steenrod square (provided there are no non-trivially framed circles in $\M(a,c)$). We remark that the second possibility arises for the disjoint union of two trefoil knots, compare \cite{JLS2}.
\end{example}

\subsection{Combinatorial proof of Theorem \ref{thm:steensquare}}
\label{subsec:moves}

\begin{definition}\label{def:ogamma}For $\Gamma$ a special graph structure with $\partial V=\emptyset$, define\[o(\Gamma):=|L|+\sum_C\left(1+F(C)+D(C)\right),\]where the sum is taken over all components $C$ of $\Gamma$ which contain a vertex.

(With this notation Definition \ref{def:sqca} can be written as $\sq^\varphi(a):=o(\Gamma_\mathcal{C}(a,\varphi))$.)
\end{definition}

We now introduce a relation on special graph structures which is not going to change the value $o(\Gamma)$. The relation is described in terms of local moves on edges. 

These local moves are given as follows.

\begin{tikzpicture}
\node at (0,0) {(a)};
\node at (1,-1) {$\varepsilon$};
\node at (3,-1) {$\varepsilon$};
\node at (1,1) {$\varepsilon+1$};
\node at (3,1) {$\varepsilon+1$};
\node[circle, draw, inner sep=0pt, minimum size=3pt] at (1,-0.6) {$ $};
\node[circle, draw, inner sep=0pt, minimum size=3pt] at (3,-0.6) {$ $};
\node[circle, draw, inner sep=0pt, minimum size=3pt] at (1,0.6) {$ $};
\node[circle, draw, inner sep=0pt, minimum size=3pt] at (3,0.6) {$ $};
\draw[->, shorten <= 2pt] (1,-0.6) -- (2,-0.6);
\draw[-, shorten >= 2pt] (1.5,-0.6) -- (3,-0.6);
\draw[->, shorten <= 2pt] (1,0.6) -- (2,0.6);
\draw[-, shorten >= 2pt] (1.5,0.6) -- (3,0.6);
\node at (4,0) {$\sim$};
\node at (5,-1) {$\varepsilon$};
\node at (7,-1) {$\varepsilon$};
\node at (5,1) {$\varepsilon+1$};
\node at (7,1) {$\varepsilon+1$};
\node[circle, draw, inner sep=0pt, minimum size=3pt] at (5,-0.6) {$ $};
\node[circle, draw, inner sep=0pt, minimum size=3pt] at (7,-0.6) {$ $};
\node[circle, draw, inner sep=0pt, minimum size=3pt] at (5,0.6) {$ $};
\node[circle, draw, inner sep=0pt, minimum size=3pt] at (7,0.6) {$ $};
\draw[-,shorten <= 2pt, shorten >= 2pt] (5,-0.6) -- (5,0.6);
\draw[-,shorten <= 2pt, shorten >= 2pt] (7,-0.6) -- (7,0.6);
\node at (8,0) {$\sim$};
\node at (9,-1) {$\varepsilon$};
\node at (11,-1) {$\varepsilon$};
\node at (9,1) {$\varepsilon+1$};
\node at (11,1) {$\varepsilon+1$};
\node[circle, draw, inner sep=0pt, minimum size=3pt] at (9,-0.6) {$ $};
\node[circle, draw, inner sep=0pt, minimum size=3pt] at (11,-0.6) {$ $};
\node[circle, draw, inner sep=0pt, minimum size=3pt] at (9,0.6) {$ $};
\node[circle, draw, inner sep=0pt, minimum size=3pt] at (11,0.6) {$ $};
\node[circle, draw, inner sep=0pt, minimum size=15pt] at (12,0) {$ $};
\draw[-,shorten <=2pt, shorten >=2pt] (9,-0.6) -- (11,0.6);
\draw[-,shorten <=2pt,shorten >=2pt] (9,0.6) -- (10,0);
\draw[-,shorten <=2pt,shorten >=2pt] (10,0) -- (11,-0.6);
\end{tikzpicture}

\begin{tikzpicture}
\node at (0,-3) {(b)};
\node at (1,-2) {$\varepsilon$};
\node at (3,-2) {$\varepsilon$};
\node at (1,-4) {$\varepsilon$};
\node at (3,-4) {$\varepsilon$};
\node[circle, draw, inner sep=0pt, minimum size=3pt] at (1,-2.4) {$ $};
\node[circle, draw, inner sep=0pt, minimum size=3pt] at (3,-2.4) {$ $};
\node[circle, draw, inner sep=0pt, minimum size=3pt] at (1,-3.6) {$ $};
\node[circle, draw, inner sep=0pt, minimum size=3pt] at (3,-3.6) {$ $};
\draw[->, shorten <= 2pt] (1,-2.4) -- (2,-2.4);
\draw[-, shorten >= 2pt] (1.5,-2.4) -- (3,-2.4);
\draw[->, shorten <= 2pt] (1,-3.6) -- (2,-3.6);
\draw[-, shorten >= 2pt] (1.5,-3.6) -- (3,-3.6);
\node at (4,-3) {$\sim$};
\node at (5,-2) {$\varepsilon$};
\node at (7,-2) {$\varepsilon$};
\node at (5,-4) {$\varepsilon$};
\node at (7,-4) {$\varepsilon$};
\node[circle, draw, inner sep=0pt, minimum size=3pt] at (5,-2.4) {$ $};
\node[circle, draw, inner sep=0pt, minimum size=3pt] at (7,-2.4) {$ $};
\node[circle, draw, inner sep=0pt, minimum size=3pt] at (5,-3.6) {$ $};
\node[circle, draw, inner sep=0pt, minimum size=3pt] at (7,-3.6) {$ $};
\draw[->, shorten <=2pt] (5,-2.4) -- (5,-3);
\draw[-, shorten >=2pt] (5,-2.6) -- (5,-3.6);
\draw[->, shorten <=2pt] (7,-3.6) -- (7,-3);
\draw[-, shorten >=2pt] (7,-3.4) -- (7,-2.4);
\end{tikzpicture}

\begin{tikzpicture}
\node at (0,-5) {(c)};
\node[circle, draw, inner sep=0pt, minimum size=3pt] at (1,-5.5) {$ $};
\node[circle, draw, inner sep=0pt, minimum size=3pt] at (3,-5.5) {$ $};
\draw[-] (1.05,-5.5) -- (2.95,-5.5) node[above,pos=0.5] {$1$};
\node at (4,-5) {$\sim$};
\node[circle, draw, inner sep=0pt, minimum size=3pt] at (5,-5.5) {$ $};
\node[circle, draw, inner sep=0pt, minimum size=3pt] at (7,-5.5) {$ $};
\node[circle, draw, inner sep=0pt, minimum size=15pt] at (6,-4.75) {$ $};
\draw[-] (5.05,-5.5) -- (6.95,-5.5) node[above,pos=0.5] {$0$};

\node at (0,-6.5) {(d)};
\node[circle, draw, inner sep=0pt, minimum size=3pt] at (1,-7) {$ $};
\node[circle, draw, inner sep=0pt, minimum size=3pt] at (3,-7) {$ $};
\draw[->] (1.05,-7) -- (2,-7);
\draw[-] (1.05,-7) -- (2.95,-7);
\node at (4,-6.5) {$\sim$};
\node[circle, draw, inner sep=0pt, minimum size=3pt] at (5,-7) {$ $};
\node[circle, draw, inner sep=0pt, minimum size=3pt] at (7,-7) {$ $};
\node[circle, draw, inner sep=0pt, minimum size=15pt] at (6,-6.25) {$ $};
\draw[<-] (6,-7) -- (6.95,-7);
\draw[-] (5.05,-7) -- (6.95,-7);

\node at (0,-8) {(e)};
\node[circle, draw, inner sep=0pt, minimum size=15pt] at (1.5,-8) {$ $};
\node[circle, draw, inner sep=0pt, minimum size=15pt] at (2.5,-8) {$ $};
\node at (4,-8) {$\sim$};
\node at (5,-8) {$\emptyset$};
\end{tikzpicture}

Note that the edges in (a), (b) and (d) are from $E-E'$, while the edges in (c) are from $E'$. The circles in (a), (c), (d) and (e) represent loops in $L$. Here (e) means that two such loops can be removed. The symbol $\varepsilon$ stands for an element of $\Z/2$ and is the sign of a vertex. Similarly, the value above an edge stands for its framing value.

We also require the following moves.

\begin{tikzpicture}
\node at (0,0) {(f)};
\node[circle,draw, inner sep=0pt, minimum size=3pt] at (1,0) {$ $};
\node[circle,draw, inner sep=0pt, minimum size=3pt] at (3,0) {$ $};
\node[circle,draw, inner sep=0pt, minimum size=3pt] at (5,0) {$ $};
\node[circle,draw, inner sep=0pt, minimum size=3pt] at (7,0) {$ $};
\node at (1,0.4) {$\varepsilon$};
\node at (3,0.4) {$\varepsilon$};
\node at (5,0.4) {$\varepsilon+1$};
\node at (7,0.4) {$\varepsilon+1$};
\draw[->] (1.05,0) -- (2,0);
\draw[-] (1.05,0) -- (2.95,0);
\draw[-] (3.05,0) -- (4.95,0) node[above,pos=0.5] {$0$};
\draw[->] (5.05,0) -- (6,0);
\draw[-] (5.05,0) -- (6.95,0);
\node at (8,0) {$\sim$};
\node[circle,draw, inner sep=0pt, minimum size=3pt] at (9,0) {$ $};
\node[circle,draw, inner sep=0pt, minimum size=3pt] at (11,0) {$ $};
\node at (9,0.4) {$\varepsilon$};
\node at (11,0.4) {$\varepsilon+1$};
\draw[-] (9.05,0) -- (10.95,0);
\end{tikzpicture}

\begin{tikzpicture}
\node at (0,-1.5) {(g)};
\node[circle,draw, inner sep=0pt, minimum size=3pt] at (1,-1.5) {$ $};
\node[circle,draw, inner sep=0pt, minimum size=3pt] at (3,-1.5) {$ $};
\node[circle,draw, inner sep=0pt, minimum size=3pt] at (5,-1.5) {$ $};
\node[circle,draw, inner sep=0pt, minimum size=3pt] at (7,-1.5) {$ $};
\node at (1,-1.1) {$\varepsilon$};
\node at (5,-1.1) {$\varepsilon$};
\node at (3,-1.1) {$\varepsilon+1$};
\node at (7,-1.1) {$\varepsilon+1$};
\draw[-] (1.05,-1.5) -- (2.95,-1.5) node [above,pos=0.5] {$0$};
\draw[-] (5.05,-1.5) -- (6.95,-1.5) node[above,pos=0.5] {$0$};
\draw[-] (3.05,-1.5) -- (4.95,-1.5);
\node at (8,-1.5) {$\sim$};
\node[circle,draw, inner sep=0pt, minimum size=3pt] at (9,-1.5) {$ $};
\node[circle,draw, inner sep=0pt, minimum size=3pt] at (11,-1.5) {$ $};
\node at (9,-1.1) {$\varepsilon$};
\node at (11,-1.1) {$\varepsilon+1$};
\draw[-] (9.05,-1.5) -- (10.95,-1.5) node[above,pos=0.5] {$0$};
\end{tikzpicture}

\begin{tikzpicture}
\node at (0,-3) {(h)};
\node[circle,draw, inner sep=0pt, minimum size=3pt] at (1,-3) {$ $};
\node[circle,draw, inner sep=0pt, minimum size=3pt] at (3,-3) {$ $};
\draw[-, bend right] (1.05,-3) to (2.95,-3);
\draw[-, bend left] (1.05,-3) to (2.95,-3);
\node at (2,-2.5) {$0$};
\node at (4,-3) {$\sim$};
\node[circle,draw, inner sep=0pt, minimum size=15pt] at (5,-3) {$ $};

\node at (0,-5) {(i)};
\node at (1,-4) {$\varepsilon$};
\node at (3,-4) {$\varepsilon$};
\node at (1,-6) {$\varepsilon+1$};
\node at (3,-6) {$\varepsilon$};
\node[circle, draw, inner sep=0pt, minimum size=3pt] at (1,-4.4) {$ $};
\node[circle, draw, inner sep=0pt, minimum size=3pt] at (3,-4.4) {$ $};
\node[circle, draw, inner sep=0pt, minimum size=3pt] at (1,-5.6) {$ $};
\node[circle, draw, inner sep=0pt, minimum size=3pt] at (3,-5.6) {$ $};
\draw[->, shorten <= 2pt] (1,-4.4) -- (2,-4.4);
\draw[-, shorten >= 2pt] (1.5,-4.4) -- (3,-4.4);
\draw[-] (1.05,-5.6) -- (2.95,-5.6);
\node at (4,-5) {$\sim$};
\node at (5,-4) {$\varepsilon$};
\node at (7,-4) {$\varepsilon$};
\node at (5,-6) {$\varepsilon+1$};
\node at (7,-6) {$\varepsilon$};
\node[circle, draw, inner sep=0pt, minimum size=3pt] at (5,-4.4) {$ $};
\node[circle, draw, inner sep=0pt, minimum size=3pt] at (7,-4.4) {$ $};
\node[circle, draw, inner sep=0pt, minimum size=3pt] at (5,-5.6) {$ $};
\node[circle, draw, inner sep=0pt, minimum size=3pt] at (7,-5.6) {$ $};
\draw[-, shorten <=2pt, shorten >=2pt] (5,-4.4) -- (5,-5.6);
\draw[->, shorten <=2pt] (7,-4.4) -- (7,-5);
\draw[-, shorten >=2pt, shorten <=2pt] (7,-4.4) -- (7,-5.6);
\end{tikzpicture}

The equivalence relation $\sim$ on special graph structures is the transitive closure of these local moves (a) - (i). Notice that the equivalence relation preserves the boundary points of a special graph structure.

\begin{lemma}
Let $\Gamma$ be a special graph structure with $\partial V=\emptyset$. Then $o(\Gamma)$ is invariant under $\sim$.
\end{lemma}

\begin{proof}
We need to check this for all local moves (a) - (i). For (c) - (h) this is obvious. The remaining cases are similar, and we will only show (i). 

Let $v,w,x$ and $y$ be vertices with $s(v)=s(w)=s(y)=s(x)+1$ and $e_1,e_2\in E''$ edges with $e_1=(v,w)$ the directed edge and $e_2=(x,y)$ the non-directed edge as in move (i). As $\partial V=\emptyset$, both edges are part of circle components $C_1$ and $C_2$ with $e_i\subset C_i$ for $i=1,2$. We can assume that $f(e)=0$ for all $e\in E'$.

Assume that $C_1$ and $C_2$ are different components. Let $a\in \Z/2$ be the number of directed edges in $C_1$ that point in a different direction than $e_1$. Also, let $b\in \Z/2$ be the number of directed edges in $C_2$ pointing in a chosen direction. The contribution of $C_1$ and $C_2$ to $o(\Gamma)$ is $a+b$.

After performing the move (i) we have a new special graph structure $\Gamma'$ where the components $C_1$ and $C_2$ are replaced by one component $C$, and we have edges $e_1'=(v,x)$ and $e_2'=(w,y)$. Note that $e_2'$ is the directed edge, and there are $a$ directed edges between $v$ and $w$ pointing in the same direction as $e_2'$, and there are $b$ directed edges between $y$ and $x$ pointing in the same direction as $e_2'$. The contribution of $C$ to $o(\Gamma')$ is therefore $1+a+b+1 = a+b$, as there are $a+b+1$ directed edges pointing in the same direction as $e_2'$ in $C$ (this includes $e_2'$). As all other components in $\Gamma'$ are unchanged from $\Gamma$, we get $o(\Gamma)=o(\Gamma')$.

Now assume that $C_1=C_2$. If we follow on from $w$ in the direction of $e_1$, we either reach $x$ or $y$ first. If we reach $x$ first, then performing the (i) move leads to a graph $\Gamma'$ where both edges $e_1'=(v,x)$ and $e_2'=(w,y)$ are again in the same component. Let $a\in \Z/2$ be the number of directed edges on the path between $w$ and $c$ which point in the same direction as $e_1$, and let $b\in \Z/2$ be the number of directed edges on the path between $y$ and $v$ which point in the same direction as $e_1$. The contribution of $C_1$ to $o(\Gamma)$ is then given by $1+a+b+1=a+b$.
Let $C$ be the new component in $\Gamma'$ containing $e_1'$ and $e_2'$. The contribution of $C$ to $\Gamma'$ is also $a+b$, as there are $1+b$ directed edges between $w$ and $v$, starting with $e_2'$, and there are $a$ directed edges on the path between $x$ and $w$ pointing in the same direction as $e_2'$. Note that since $s(x)\not=s(w)$ there are indeed the same number of directed edges in a chosen direction between $x$ and $w$, which in this case is $a$. Hence we get again $o(\Gamma)=o(\Gamma')$.

If we reach $y$ first, then performing the (i) move leads to a graph $\Gamma'$ where the edges $e_1'=(v,x)$ and $e_2'=(w,y)$ are now in different components. This case is now basically just the reverse of the first case where we changed from two circles to one. The same argument thus shows $o(\Gamma)=o(\Gamma')$.
\end{proof}

\begin{proof}[Proof of Theorem \ref{thm:steensquare}]
Let the cocycle $\varphi\in C^k(\cC;\Z/2)$ be represented by objects $c_1,\ldots,c_m\in \Ob(\cC)$ satisfying $|c_i|=k$ for all $i=1,\ldots,m$. We begin by showing that $\sq^\varphi$ is a cocycle. Let $u\in \Ob(\cC)$ satisfy $|u|=k+3$. Then
\[
 \partial \M(u,c_i) = \coprod_a\M(a,c_i)\times \M(u,a) \sqcup \coprod_b \M(b,c_i)\times \M(u,b)
\]
where the first union is over all $a\in \Ob(\cC)$ with $|a|=k+2$ and the second union over all $b\in \Ob(\cC)$ with $|b|=k+1$. We associate a special graph structure $\Gamma(u,c_i)$ to this as follows. The vertices are given by
\[
 \coprod_{a,b} \M(b,c_i)\times \M(a,b)\times \M(u,a)
\]
with $s(C,B,A) = s(C)+s(B)+s(A)$. As edges in $E'$ we use $J\times \{A\}\subset \M(a,c_i)\times \M(u,a)$, where $J\subset \M(a,c_i)$ is an interval component. We set $f(J\times \{A\})=f(J)$. The edges in $E-E'$ are given by the $\{C\}\times I \subset \M(b,c_i)\times \M(u,b)$ with $I$ an interval component in $\M(u,b)$. Furthermore, if such an interval $\{C\}\times I$ satisfies $s(C)+f(I)=0$, we add a loop in $L$. Finally, any non-trivially framed circle $\{C\}\times S^1\subset \M(b,c_i)\times \M(u,b)$ and $S^1\times \{A\}\subset \M(a,c_i)\times \M(u,a)$ gives rise to a loop in $L$. Note that $E''=\emptyset$.

With this definition, we get $o(\Gamma(u,c_i))=0$, because loops and framings were chosen so that we get exactly the sum from the compatibility condition for $\partial\M(u,c_i)$.

Also form the special graph structure $\Gamma(u,\varphi)$ as the disjoint union of the $\Gamma(u,c_i)$

We want to replace edges coming from $\{C\}\times I$ with edges coming from the combinatorial boundary matching $\mathcal{C}$. Note that $C\in \M(b,c_i)$ is matched to some $C'\in \M(b,c_j)$, and we therefore also have an edge coming from $\{C'\}\times I$. We can then use a move (a) to replace these two edges with edges coming from the boundary matching. If $s(C)=s(C')$, we have $s(C)+f(I)=0$ if and only if $s(C')+f(I)$ so possible extra loops coming from this condition cancel each other. If $s(C)\not =s(C')$ this condition creates exactly one extra loop, which cancels the loop coming from move (a).

Replacing all the loops $\{C\}\times I$ this way leads to the special graph structure
\[
 \coprod_a \Gamma_\mathcal{C}'(a,\varphi) \times \M(u,a)
\]
where $\Gamma'_\mathcal{C}(a,\varphi)\times \{A\}$ only differs from $\Gamma_\mathcal{C}(a,\varphi)$ in that the sign functions for vertices differ by $s(A)$. We thus get that $o(\Gamma'_\mathcal{C}(a,\varphi)\times \{A\})=o(\Gamma_\mathcal{C}(a,\varphi)$. Furthermore,
\begin{align*}
 0 &= o\left(\coprod_a \Gamma_\mathcal{C}'(a,\varphi) \times \M(u,a)\right) \\
 &=\sum_a o(\Gamma_\mathcal{C}(a,\varphi))\cdot [u:a] \\
 &=\delta(\sq^\varphi)(u).
\end{align*}
Therefore $\sq^\varphi$ is a cocycle.

To see that the cohomology class of $\sq^\varphi$ does not depend on the combinatorial boundary matching, first consider the order of a directed edge $(C,C')$ with $C\in \M(b,c_i)$ and $C'\in \M(b,c_j)$ for some $b\in \Ob(\cC)$ with $|b|=k+1$. Then $o(\Gamma_{\mathcal{C}'}(a,\varphi))$ differs from $o(\Gamma_\mathcal{C}(a,\varphi))$ in that every occurance of $(C,B)\in \M(b,c_j)\times \M(a,b)$ leads to a $+1$. In other words, $o(\Gamma_{\mathcal{C}'}(a,\varphi)) + o(\Gamma_{\mathcal{C}}(a,\varphi)) = [a:b]\in \Z/2$. This is true for every $a\in \Ob(\cC)$ with $|a|=k+2$, so this difference is compensated by adding $\delta(b)$ to $\sq^\varphi$.

If we are given matchings $(C_1,C_2)$ and $(C_3,C_4)$, we can replace this by matchings $(C_1,C_3)$ and $(C_2,C_4)$ which in the special graph structures correspond to moves (a) or (b) depending on signs of the points involved, again for all combinations with $A\in \M(a,b)$. So any possible circles arising from move (a) would again be compensated by $\delta(b)$. As we can go between any boundary matching using these steps, we see that the cohomology class does not depend on it.

If we are given two cocycles $\varphi_1$ and $\varphi_2$ we now show $\sq^{\varphi_1+\varphi_2}=\sq^{\varphi_1}+sq^{\varphi_2}$. Taking respective combinatorial matchings $\mathcal{C}_1$ and $\mathcal{C}_2$ for $\varphi_1$ and $\varphi_2$, we will have in general that some of the objects $c_i$ are used for representing both $\varphi_1$ and $\varphi_2$, so are not needed for $\varphi_1+\varphi_2$. We can begin by letting $\mathcal{C}$ be the disjoint union of $\mathcal{C}_1$ and $\mathcal{C}_2$, and setting
\[
 \Gamma_\mathcal{C}(a,\varphi_1+\varphi_2) = \Gamma_{\mathcal{C}_1}(a,\varphi_1) \sqcup \Gamma_{\mathcal{C}_2}(a,\varphi_2).
\]
Let $c_i$ be representing both $\varphi_1$ and $\varphi_2$ and assume that $I\subset \M(a,c_i)$ is an interval with endpoints $(C,B)$ and $(C',B')$. Let $C$ be matched to $C_l$ by $\mathcal{C}_l$ and $C'$ be matched to $\mathcal{C}_l$ for $l=1,2$. In $\Gamma_\mathcal{C}(a,\varphi_1+\varphi_2)$ the interval $I$ occurs twice, and we can use the moves (a), (b) and (i) to change $\Gamma_\mathcal{C}(a,\varphi_1+\varphi_2)$ without changing $o(\Gamma_\mathcal{C}(a,\varphi_1+\varphi_2))$.
\begin{center}
\begin{tikzpicture}
\node[circle,draw, inner sep=0pt, minimum size=3pt] at (1,0) {$ $};
\node[circle,draw, inner sep=0pt, minimum size=3pt] at (2.5,0) {$ $};
\node[circle,draw, inner sep=0pt, minimum size=3pt] at (4,0) {$ $};
\node[circle,draw, inner sep=0pt, minimum size=3pt] at (5.5,0) {$ $};
\node at (1,-0.4) {$(C_1,B)$};
\node at (2.5,-0.4) {$(C,B)$};
\node at (3.25,0.25) {$I$};
\node at (4,-0.4) {$(C',B')$};
\node at (5.5,-0.4) {$(C'_1,B')$};
\draw[-,blue] (1.05,0) -- (2.45,0);
\draw[-] (2.55,0) -- (3.95,0);
\draw[-,blue] (4.05,0) -- (5.45,0);
\draw[-] (0.55,0) -- (0.95,0);
\draw[-] (5.55,0) -- (5.95,0);
\node at (6.5,0.7) {$\sim$};

\node[circle,draw, inner sep=0pt, minimum size=3pt] at (7.5,0) {$ $};
\node[circle,draw, inner sep=0pt, minimum size=3pt] at (9,0) {$ $};
\node[circle,draw, inner sep=0pt, minimum size=3pt] at (10.5,0) {$ $};
\node[circle,draw, inner sep=0pt, minimum size=3pt] at (12,0) {$ $};
\node at (7.5,-0.4) {$(C_1,B)$};
\node at (9,-0.4) {$(C,B)$};
\node at (9.75,0.25) {$I$};
\node at (10.5,-0.4) {$(C',B')$};
\node at (12,-0.4) {$(C'_1,B')$};
\draw[->,blue] (7.5,0.05) -- (7.5,0.75);
\draw[-,blue] (7.5,0.75) -- (7.5,1.45);
\draw[-] (9.05,0) -- (10.45,0);
\draw[->,blue] (10.5,0.05) -- (10.5,0.75);
\draw[-,blue] (10.5,0.75) -- (10.5,1.45);
\draw[-] (7.05,0) -- (7.45,0);
\draw[-] (12.05,0) -- (12.45,0);

\node[circle,draw, inner sep=0pt, minimum size=3pt] at (1,1.5) {$ $};
\node[circle,draw, inner sep=0pt, minimum size=3pt] at (2.5,1.5) {$ $};
\node[circle,draw, inner sep=0pt, minimum size=3pt] at (4,1.5) {$ $};
\node[circle,draw, inner sep=0pt, minimum size=3pt] at (5.5,1.5) {$ $};
\node at (1,1.9) {$(C_2,B)$};
\node at (2.5,1.9) {$(C,B)$};
\node at (3.25,1.25) {$I$};
\node at (4,1.9) {$(C',B')$};
\node at (5.5,1.9) {$(C'_2,B')$};
\draw[-,blue] (1.05,1.5) -- (2.45,1.5);
\draw[-] (2.55,1.5) -- (3.95,1.5);
\draw[-,blue] (4.05,1.5) -- (5.45,1.5);
\draw[-] (0.55,1.5) -- (0.95,1.5);
\draw[-] (5.55,1.5) -- (5.95,1.5);

\node[circle,draw, inner sep=0pt, minimum size=3pt] at (7.5,1.5) {$ $};
\node[circle,draw, inner sep=0pt, minimum size=3pt] at (9,1.5) {$ $};
\node[circle,draw, inner sep=0pt, minimum size=3pt] at (10.5,1.5) {$ $};
\node[circle,draw, inner sep=0pt, minimum size=3pt] at (12,1.5) {$ $};
\node at (7.5,1.9) {$(C_2,B)$};
\node at (9,1.9) {$(C,B)$};
\node at (9.75,1.25) {$I$};
\node at (10.5,1.9) {$(C',B')$};
\node at (12,1.9) {$(C'_2,B')$};
\draw[->,blue] (9,0.05) -- (9,0.75);
\draw[-,blue] (9,0.75) -- (9,1.45);
\draw[-] (9.05,1.5) -- (10.45,1.5);
\draw[->,blue] (12,0.05) -- (12,0.75);
\draw[-,blue] (12,0.75) -- (12,1.45);
\draw[-] (7.05,1.5) -- (7.45,1.5);
\draw[-] (12.05,1.5) -- (12.45,1.5);
\end{tikzpicture}
\end{center}
The square in the middle involving both intervals $I$ contributes $0$ to $o(\Gamma_\mathcal{C}(a,\varphi_1+\varphi_2))$ and can be removed. Note that the directed edges may be in different places, but the square needs to have directed edges. We can now remove $C$ and $C'$ from $\mathcal{C}$ and match $C_1$ with $C_2$ and $C_1'$ with $C_2'$. After finitely many steps we get the right special graph structure for $\sq^{\varphi_1+\varphi_2}$, and as the value $o(\Gamma_\mathcal{C}(a,\varphi_1+\varphi_2))$ never changed, we have $\sq^{\varphi_1+\varphi_2}=\sq^{\varphi_1}+sq^{\varphi_2}$.

It remains to show that $\sq^\varphi$ is a coboundary for $\varphi$ a coboundary. So let $d\in \Ob(\cC)$ satisfy $|d|=k-1$ and let $c_1,\ldots,c_m$ be the objects satisfying $|c_i|=k$ and $\M(c_i,d)\not=\emptyset$. Let us first assume that for each $i$, $\M(c_i,d)$ contains exactly one point.

If $a\in \Ob(\cC)$ satisfies $|a|=k+2$, consider 
\[
 \partial\M(a,d) = \coprod_b \M(b,d) \times \M(a,b) \sqcup \coprod_{c'} \M(c',d) \times \M(a,c')
\]
where the first disjoint union is over all objects $b$ with $|b|=k+1$ and the second over objects $c'$ with $|c'|=k$. Note that we can restrict the second disjoint union to the objects $c_1,\ldots,c_m$. Define a special graph structure $\Gamma(a,d)$ as follows. Vertices are given by
\[
 \coprod_{b, i} \M(c_i,d)\times \M(b,c_i)\times \M(a,b)
\]
and we set $s(D,C,B)=s(D)+s(C)+s(B)$. As edges in $E'$ we use $\{D\}\times I \subset \M(c_i,d)\times \M(a,c_i)$, where $I$ is an interval component. We also set $f(\{D\}\times I)=s(D)+f(I)$. The edges in $E-E'$ are given by $J\times \{B\}\subset \M(b,d)\times \M(a,b)$, where $J$ is an interval component. For every such interval $J\times \{B\}$ we add a loop to $L$ provided that $f(J)=0$. Finally, any non-trivially framed circle $\{D\}\times S^1\subset \M(c_i,d)\times \M(a,c_i)$ and $S^1\times \{B\}\subset \M(b,d)\times \M(a,b)$ gives rise to a loop in $L$. The set of directed edges $E''=\emptyset$.

Note that $\Gamma(a,d)$ is very similar to the construction of $\Gamma(u,c_i)$ at the beginning of the proof, but the roles of $E'$ and $E-E'$ are reversed. Nevertheless, we get again that $o(\Gamma(a,d))=0$ from the compatibility condition of $\partial\M(a,d)$.

Intervals $J\subset \M(b,d)$ have endpoints $(D_i,C)\in \M(c_i,d)\times \M(b,c_i)$ and $(D_j,C')\in \M(c_j,d)\times \M(b,c_j)$ for some $i,j$. As we assume that each $\M(c_i,d)$ has at most one point, we can use these intervals to define the combinatorial boundary matching $\mathcal{C}$. Namely, we match $C$ and $C'$. If $s(C)=s(C')$, then $s(D_i)\not=s(D_j)$ and we choose the ordering $(C,C')$ if and only if $s(D_i)=0$.

Consider an edge $\{D_i\}\times I \subset \M(c_i,d)\times \M(a,c_i)$ in $\Gamma(a,d)$. There is a corresponding edge $I$ in $\M_\mathcal{C}(a,\delta(d))$, but if $s(D_i)=1$, the vertices have different sign, and the contribution to $o(\Gamma(a,d))$ and $o(\Gamma_\mathcal{C}(a,\delta(d)))$ differ by $1$. However, we can perform a move of type (a) in $\Gamma(a,d)$ as follows:
\begin{center}
\begin{tikzpicture}
\node[circle,draw, inner sep=0pt, minimum size=3pt] at (1,0) {$ $};
\node[circle,draw, inner sep=0pt, minimum size=3pt] at (2.5,0) {$ $};
\node[circle,draw, inner sep=0pt, minimum size=3pt] at (1,1) {$ $};
\node[circle,draw, inner sep=0pt, minimum size=3pt] at (2.5,1) {$ $};
\node at (1,-0.4) {$\varepsilon$};
\node at (1,1.4) {$\varepsilon+1$};
\node at (2.5,1.4) {$\varepsilon$};
\node at (2.5,-0.4) {$\varepsilon+1$};
\draw[-,blue] (1,0.05) -- (1,0.95);
\node at (0.3,0.5) {$J\times\{B\}$};
\draw[-,blue](2.5,0.05) -- (2.5,0.95);
\node at (3.3,0.5) {$J'\times \{B'\}$};
\draw[-] (1.05,1) -- (2.45,1);
\node at (1.75,0.75) {$\{D_i\}\times I$};
\draw[-] (0.55,0) -- (0.95,0);
\draw[-] (2.55,0) -- (2.95,0);
\node at (4.5,0.5) {$=$};
\node[circle,draw, inner sep=0pt, minimum size=3pt] at (5.5,0) {$ $};
\node[circle,draw, inner sep=0pt, minimum size=3pt] at (7,0) {$ $};
\node[circle,draw, inner sep=0pt, minimum size=3pt] at (5.5,1) {$ $};
\node[circle,draw, inner sep=0pt, minimum size=3pt] at (7,1) {$ $};
\node at (5.5,-0.4) {$\varepsilon$};
\node at (5.5,1.4) {$\varepsilon$};
\node at (7,1.4) {$\varepsilon+1$};
\node at (7,-0.4) {$\varepsilon+1$};
\draw[-] (5.05,0) -- (5.45,0);
\draw[-] (7.05,0) -- (7.45,0);
\draw[-] (5.55,1) -- (6.95,1);
\draw[-,blue] (5.55,0) -- (6.95,1);
\draw[-,blue] (5.55,1) -- (6.2,0.55);
\draw[-,blue] (6.3,0.45) -- (6.95,0);
\node at (8,0.5) {$\sim$};
\node[circle,draw, inner sep=0pt, minimum size=3pt] at (9,0) {$ $};
\node[circle,draw, inner sep=0pt, minimum size=3pt] at (10.5,0) {$ $};
\node[circle,draw, inner sep=0pt, minimum size=3pt] at (9,1) {$ $};
\node[circle,draw, inner sep=0pt, minimum size=3pt] at (10.5,1) {$ $};
\node at (9,-0.4) {$\varepsilon$};
\node at (9,1.4) {$\varepsilon$};
\node at (10.5,1.4) {$\varepsilon+1$};
\node at (10.5,-0.4) {$\varepsilon+1$};
\draw[-] (8.55,0) -- (8.95,0);
\draw[-] (10.55,0) -- (10.95,0);
\draw[-] (9.05,1) -- (10.45,1);
\draw[->,blue] (9,0.05) -- (9,0.55);
\draw[-,blue] (9,0.55) -- (9,0.95);
\draw[->,blue] (10.5,0.05) -- (10.5,0.55);
\draw[-,blue] (10.5,0.55) -- (10.5,0.95);
\end{tikzpicture}
\end{center}
The right graph has $I$ as in $\Gamma_\mathcal{C}(a,\delta(d))$. Note that $J$ and $J'$ have as endpoints $(D_j,C')$ and $(D_k,C'')$ at the bottom. If $s(D_j)=0=s(D_k)$, then the blue directed edges are exactly as the boundary matching prescribes, and the contribution of $s(D_i)=1$ corresponds to the $1$ coming from the directed edge pointing with an orientation. If $s(D_k)=1$, we have to do another such move to get the edge unoriented again. Repeating this throughout will show that the contributions of $s(D_i)$ to $o(\Gamma(a,d))$ is the same as the contribution of the edges coming from the boundary matching via directed edges to $o(\Gamma_\mathcal{C}(a,\delta(d)))$.

There are still some contributions to $o(\Gamma(a,d))$ coming from loops which are not present in $o(\Gamma_\mathcal{C}(a,\delta(d)))$, namely coming from edges $J\times \{B\}\subset \M(b,d)\times \M(a,b)$ with $f(J)=0$ and non-trivially framed circles $S^1\times \{B\} \subset \M(b,d)\times \M(a,b)$. But for each $B$ we only need to check how many contributions from $\M(b,d)$ there are and compensate with $\delta(b)$. This shows that $\sq^{\delta(d)}$ differs from $0$ by a coboundary.

Finally, assume that a $\M(c_i,d)$ contains two different points $X_1,X_2$. We can form $\Gamma(a,d)$ as before. Any interval $I\subset \M(a,c_i)$ leads to two edges $\{X_1\}\times I$ and $\{X_2\}\times I$ in $\Gamma(a,d)$ contributing $s(X_1)+s(X_2)$ to $o(\Gamma(a,d))$. We can replace these two edges by edges between its endpoints using a move (a) without changing $o(\Gamma(a,d))$. We can do this with any pairs of points in $\M(c_i,d)$ until there is only one point (or none) in each $\M(c_i,d)$ not paired this way.
We can then use these remaining points to obtain a combinatorial boundary matching as in the previous case. The same argument as before shows that $\sq^{\delta(d)}$ differs from $0$ by a coboundary.
\end{proof}

It is a great effort to turn the topological data of a framed flow category into the combinatorial data of a 1-flow category, and then to verify the formula for the second Steenrod square, given a combinatorial boundary matching.  Nevertheless, it can be done and was carried in \cite{LipSarSq} for the Lipshitz-Sarkar framed flow category and also in \cite{JLS} in general.  Treating this (considerable) work as a black box, one is able to make the following remark.

\begin{remark}\label{rem:topsquare}
Theorem \ref{thm:steensquare} can alternatively be proved using the topological Steenrod square as follows. First we can truncate the framed 1-flow category $\cC$ so that all objects are in degree $k-1$ to $k+3$. This does not affect the cohomology groups $H^k(\cC;\Z/2)$ and $H^{k+2}(\cC;\Z/2)$. We now want to describe a framed flow category $\cC'$ which extends $\cC$ and so that the cohomology in degree $k$ and $k+2$ remains the same

We use the same objects for $\cC'$ as for $\cC$, and it is straightforward to frame the $0$- and $1$-dimensional moduli spaces to extend $\cC$. The $2$-dimensional moduli spaces of $\cC'$ can then be given some framing because of the compatibility condition that $\cC$ satisfies.  For all pairs of objects $a$ and $e$ with $|a|=k+3$ and $|e|=k-1$, this gives a framed surface $S(a,e)$ consisting of all broken flowlines between $a$ and $e$.
If the framing is trivial we can extend $S(a,e)$ to a $3$-dimensional framed moduli space $\M(a,e)$, and we are done.

If the framing is non-trivial, we add an extra object $c_{a,e}$ of grading $b+1$ together with non-trivially framed circles as the moduli spaces $\M(a,c_{a,e})$ and $\M(c_{a,e},e)$.  This then adds a non-trivially framed component to $S(a,e)$, hence $S(a,e)$ becomes trivially framed overall.  We are then able to choose a $3$-dimensional framed moduli space $\M(a,e)$ with $\partial \M(a,e) = S(a,e)$, and the resulting framed flow category $\cC'$ has a well defined second Steenrod square from \cite{JLS}.
Furthermore, the formula agrees exactly with the formula for $\cC$.
\end{remark}

As a neat application of Theorem \ref{thm:steensquare}, we finish this section with an example showing that not every framed $1$-flow category extends to a framed flow category.

\begin{example}\label{ex:adem}
Let $\cC$ have four objects $a,b,d,e$ with $k+3=|a|=|b|+1=|d|+3=|e|+4$ and let $\M(a,b)$ and $\M(d,e)$ both consist of two points with positive sign. Furthermore, let $\M(b,d)$ contain a non-trivially framed circle. This satisfies the compatibility condition and thus defines a framed $1$-flow category. But notice that
\[
 \Sq^1 \Sq^2 \Sq^1(e) = a \not= 0
\]
while $\Sq^2(e) = 0$. Therefore $\cC$ cannot be turned into a framed flow category, as the resulting space would violate the Adem relation $\Sq^1 \Sq^2 \Sq^1 = \Sq^2 \Sq^2$. Note that adding an object $c$ with $|c|=k+2$ as in Remark \ref{rem:topsquare} establishes exactly this Adem relation.
\end{example}

\section{Morse moves in 1-flow categories}\label{sec:handle}

In \cite{JLS2} and \cite{ALPODS2} the authors, together with Dan Jones, defined a series of \emph{flow category moves} -- called \emph{handle slides}, \emph{handle cancellation}, and \emph{Whitney trick} -- for modifying and algorithmically simplifying a framed flow category. Section \ref{sec:handle} is dedicated to an explicit description of the restrictions of these moves to 1-flow categories, with a mind to their application in Section \ref{sec:snf_for_flow_categories} as part of the new algorithm for computing Steenrod squares.

\subsection{Combinatorial handle slides}
The handle slides of \cite{ALPODS2} have obvious analogues in framed $1$-flow categories.

\begin{definition}
 Let $(\cC,s,f)$ be a framed $1$-flow category, and let $x,y\in \Ob(\cC)$ satisfy $|x|=|y|$. For $\varepsilon \in \{0,1\}$ let $(\cC^\varepsilon_S,s',f')$ be the following framed $1$-flow category:
\begin{itemize}
 \item The objects of $\cC^\varepsilon_S$ are in one-to-one correspondence with the objects of $\cC$, and for $a\in \Ob(\cC)$ we write $a'\in \Ob(\cC^\varepsilon_S)$ for the corresponding object. Furthermore, $|a'|=|a|$.
 \item The moduli spaces are given by
\begin{align*}
 \M(x',b') &= \M(x,b)\sqcup \M(y,b)\\
 \M(a',y') &= \M(a,x)\sqcup \M(a,y)
\end{align*}
If $|a|=|x|+1=|c|+2$ we also have
\[
 \M(a',c') = \M(a,c) \sqcup \M(y,c) \times [0,1]\times \M(a,x).
\]
\item The sign assignment $s'$ is defined by
\[
 s'(X) = \left\{ 
\begin{array}{cl}
 s(X) & \mbox{if }X\in \M(a,b)\subset \M(a',b')\\
 \varepsilon + s(X) & \mbox{if }X\in \M(y,b)\subset \M(x',b')\\
 1+\varepsilon + s(X) & \mbox{if }X \in \M(a,x) \subset \M(a',y')
\end{array}
\right.
\]
\item The framing $f'$ is defined by
\[
 f'(I) = \left\{
\begin{array}{cl}
 f(I) & \mbox{if } I \subset \M(a,b)\subset \M(a',b')\\
 f(I) & \mbox{if } I \subset \M(y,b)\subset \M(x',b')\\
 1+\varepsilon+f(I) & \mbox{if } I \subset \M(a,x) \subset \M(a',y')\\
 1+\varepsilon+s(B) & \mbox{if } I = \{B\}\times [0,1]\times \{A\}\subset \M(y,b)\times [0,1]\times \M(a,x)
\end{array}
\right.
\]
\end{itemize}
\end{definition}

We also call the case $\varepsilon=0$ a $(+)$-slide and the case $\varepsilon=1$ a $(-)$-slide. We need to check that $(\cC^\varepsilon_S,s',f')$ is indeed a framed $1$-flow category.

\begin{lemma}
 Let $(\cC,s,f)$ be a framed $1$-flow category, and let $x,y\in \Ob(\cC)$ satisfying $|x|=|y|$. For $\varepsilon\in \{0,1\}$ the collection $(\cC^\varepsilon_S,s',f')$ is a framed $1$-flow category.
\end{lemma}

\begin{proof}
 It is easy to see that $s'$ is a sign assignment. We need to check the compatibility condition for $\partial\M(a',d')$ for any $a',d'\in \Ob(\cC)$ with $|a'|=|d'|+3$. There are essentially four cases depending on $|a'|-|x'|$.

If $|a'|=|x'|$ we have $\partial\M(a',d')=\partial\M(a,d)$ with the same framings, unless $a'=x'$. In that case
\[
 \partial\M(x',d') = \partial \M(x,d) \sqcup \partial\M(y,d)
\]
and the framing values in the relevant circle components are the same as for $f$.

If $|a'|=|x'|+1$, every interval $J\subset \M(y,d)$ and $A\in \M(a,x)$ leads to two intervals $J \subset \M(y',d')$ and $J'\subset \M(x',d')$ and two points $A\in \M(a',x')$ and $A'\in \M(a',y')$. Let the endpoints of $J$ be given by $(C,B)\in \M(c,d)\times \M(y,c)$ and $(\tilde{C},\tilde{B})\in \M(\tilde{c},d)\times \M(y,\tilde{c})$. We then also get the intervals $I_{B,A}=\{B\}\times [0,1]\times \{A\} \subset \M(a',c')$ and $I_{\tilde{B},A}\subset \M(a',\tilde{c}')$.

We then get a component $C_{J,A}$ in $\partial\M(a',d')$ from the four intervals $J'\times\{A\}$, $J\times \{A'\}$, $\{C\}\times I_{B,A}$, and $\{\tilde{C}\}\times I_{\tilde{B},A}$. Then
\begin{align*}
 \tilde{f}(C_{J,A}) &= f'(J')+f'(J)+1+s'(C)+f'(I_{B,A})+1+s'(\tilde{C})+f'(I_{\tilde{B},A}) \\
 &= f(J)+f(J)+s(C)+s(\tilde{C})+1+\varepsilon+s(B)+1+\varepsilon+s(\tilde{B})\\
 &=1.
\end{align*}
This component therefore contributes nothing to the compatibility condition. Similarly, if $A\in \M(a,x)$ and $S\subset \M(y,d)$ a circle, we get two extra circles $S\times \{A'\}$ and $S'\times \{A\}$ which cancel each other in the compatibility condition. The remaining components of $\partial\M(a',d')$ are present in $\partial\M(a,d)$ with the same contributions.

The case $|a'|=|x'|+2$ is similar to the previous case, we now get new components $C_{C,I}$ made up of four intervals for every $C\in \M(y,d)$ and every interval $I\subset \M(a,x)$. A straightforward check shows $\tilde{f}(C_{C,J})=1$ as in the previous case.

For $|a'|=|x'|+3$ one can give an argument similar to the first case. We omit the details.
\end{proof}

\begin{lemma}
 Let $(\cC,s,f)$ be a framed $1$-flow category and $(\cC^\varepsilon_S,s',f')$ be the framed $1$-flow category obtained by a handle slide. Then there is an isomorphism 
\[
\Phi\colon H^\ast(\cC;\Z/2)\to H^\ast(\cC^\varepsilon_S;\Z/2)
\]
which commutes with $\Sq^2$.
\end{lemma}

\begin{proof}
From the definition we see that
\[
 [a':y'] = [a:y]+(-1)^{\varepsilon+1}[a:x]
\]
and
\[
 [x':c'] = [x:c]+(-1)^{\varepsilon}[y:c]
\]
and all other incidence numbers remain the same. It is now easy to check that we get an isomorphism of chain complexes $\Phi\colon C_\ast(\cC^\varepsilon_S)\to C_\ast(\cC)$ which differs from the identity only in that $\Phi(x')=x+(-1)^\varepsilon y$. The induced map on cochain complexes $\Phi\colon C^\ast(\cC;\Z/2) \to C^\ast(\cC^\varepsilon_S;\Z/2)$ satisfies $\Phi(y)=x'+y'$.

Now let $\varphi\in C^n(\cC;\Z/2)$ be a cocycle with $n=|x|=|y|$, and let $\varphi'=\Phi(\varphi)$ the corresponding cocycle in $C^n(\cC^\varepsilon_S;\Z/2)$. If $y$ is not among the objects representing $\varphi$, it is clear that the Steenrod square cannot change. 
Therefore assume that $y$ is among the objects representing $\varphi$. If $x$ is among the objects representing $\varphi$, then $x'$ is not among the objects representing $\varphi'$ and vice versa. Depending on whether $x$ is among the objects representing $\varphi$, the moduli space $\M(a,x)$ either contributes to $\sq^\varphi$ or $\sq^{\varphi'}$. Furthermore, if $\varepsilon=1$, the relevant framing values remain the same and it is clear that $\Phi$ commutes with Steenrod squares.

If $\varepsilon=0$, then intervals in $\M(a,x)$ have different framing when considered as intervals in $\M(a',y')$. Also, points in $\M(b,x)$ change their sign when considered in $\M(b',y')$. Assume first that $x$ is among the objects representing $\varphi$. Consider a combinatorial boundary matching $\mathcal{C}$ for $\varphi$. We can assume that whenever an element $B\in \M(b,x)$ is matched to an element $B'\in \M(b,c_i)$ with $c_i\not=x$ and $s(B')=s(B)$, the matching is ordered as $(B,B')$. For the combinatorial boundary matching $\mathcal{C}'$ we then keep the matching, although we now consider $B\in \M(b',y')$. 
Note that the sign has changed so the matching is no longer considered ordered. If $s(B')\not=s(B)$, we get $s'(B')=s'(B)$, and we choose the ordering $(B,B')$ in $\mathcal{C}'$. With this ordering convention we can ensure that the number of directed edges compensates for the change of framing of intervals in $\M(a,x)$.

We need to be more careful if $B,B'\in \M(b,x)$ are matched in $\mathcal{C}$. If we keep the identical matching in $\mathcal{C}'$, we cannot compensate for the extra $+1$ from each interval. However, since any $A\in \M(a,b)$ leads to endpoints $(B,A)$ and $(B',A)$, this extra contribution is compensated by the coboundary $\delta(b)$.

The cases where $|x|=n+1$ or $n+2$ are easier and we omit them. 
\end{proof}

\subsection{Combinatorial Whitney trick}
In \cite{JLS2} the Whitney trick in a flow category was described. We now give the analogue for a $1$-flow category.

\begin{definition}\label{def:whitney}
Let $(\cC,s,f)$ be a framed $1$-flow category, and let $x,y\in \Ob(\cC)$ satisfy $|x|=|y|+1$. Assume that $\M(x,y)$ contains two points $P,M$ with $s(P)=0$ and $s(M)=1$. Then let $\cC',s',f')$ be the framed $1$-flow category given as follows:
\begin{itemize}
\item We have $\Ob(\cC')=\Ob(\cC)$, and we write $a'\in \Ob(\cC')$ for $a\in \Ob(\cC)$ with $|a'|=|a|$.
\item The $0$-dimensional moduli spaces are unchanged except for $\M(x',y')=\M(x,y)=\{P,M\}$, furthermore $s'(A)=s(A)$ for every $A$ in a $0$-dimensional moduli space.
\item The $1$-dimensional moduli spaces are unchanged except for $\M(a',y')$ and $\M(x',b')$, which are obtained from $\M(a,y)$ and $\M(x,b)$ respectively, by identifying endpoints of intervals of the form $(P,A)\sim (M,A)$ for $\M(a,y)$, and endpoints $(B,P)\sim (B,M)$ for $\M(x,b)$.
\item Each component $C$ of $\M(a',y')$ is a quotient space of finitely many components $C_1,\ldots,C_k$ of $\M(a,y)$. The new framing value is given by
\[
 f'(C)=\sum_{i=1}^k f(C_i)\in \Z/2.
\]
\item Each component $C$ of $\M(x',b')$ is a quotient space of finitely many components $C_1,\ldots,C_k$ of $\M(x,b)$. The new framing value is given by
\[
 f'(C)=\sum_{i=1}^k f(C_i)+\chi \in \Z/2
\]
where $\chi$ is the number of gluings $(B,P)\sim (B,M)$ with $s(B)=1$.
\end{itemize}
\end{definition}

\begin{remark}
By \cite[Prop.3.4]{JLS2} one should expect that $\chi$ is the number of gluings $(B,P)\sim (B,M)$ with $s(B)=0$, and \cite[Prop.3.3]{JLS2} suggests one should add $1$ for every gluing that is taking place. This is however the result of an unfortunate framing convention used in \cite[\S 3.3]{JLS2}. If we mirror the roles of $P$ and $M$ in \cite[Figure 12]{JLS2} one gets a convention in line with Definition~\ref{def:whitney}. The reader may want to check that this alternative convention also leads to a framing on $\cC'$ below.
\end{remark}

\begin{lemma}
The tuple $(\cC',s',f')$ is a framed $1$-flow category. Furthermore, there is an isomorphism 
\[
\Phi\colon H^\ast(\cC;\Z/2)\to H^\ast(\cC';\Z/2)
\]
which commutes with $\Sq^2$.
\end{lemma}

\begin{proof}
It is clear that $s'$ is a sign assignment, so we need to check the compatibility condition for $\partial \M(a',d')$. 

First consider the case where $|d'|=|y'|$. If $d'\not=y'$, the compatibility condition is clear.
Consider $(P,B,A), (M,B,A)\in \partial \M(a,y)$. Then $(P,B,A)$ is a boundary point of an interval $\{P\}\times I \subset \M(x,y)\times \M(a,x)$ and of an interval $J_1\times \{A\}\subset \M(b,y)\times \M(a,b)$ for some object $b$ with $|b|=|a|-1$. Similarly, $(M,B,A)$ is the boundary point of an interval $\{M\}\times I$ and of an interval $J_2\times \{A\}$.
The interval $I$ has a second endpoint $(B',A')$, leading to two more points $(P,B',A')$ and $(M,B',A')$, which also are endpoints of intervals $J_3\times \{A'\}$ and $J_4\times \{A'\}$, respectively. In $\partial \M(a',y')$ the intervals $\{P,M\}\times I$ are no longer present, and $J_1\times \{A\}$, $J_2\times \{A\}$ have been glued together, as well as $J_3\times \{A'\}$, $J_4\times \{A'\}$ which have also been glued together.

In $\partial\M(a,y)$ the two intervals $\{P,M\}\times I$ are either in the same or in different components. If they are in the same component, then $\partial \M(a',y')$ has one more component than $\partial\M(a,y)$, and if they are in different components, then $\partial\M(a',y')$ has one less component than $\M(a,y)$. But notice that $\{M\}\times I$ contributes an extra summand $1$ to the compatibility condition, which offsets the change in number of components. Therefore the compatibility condition is also satisfied at $\partial \M(a',y')$.

The cases where $|a'|=|x'|+1$ and $a'=x'$ are similar, and will be omitted.

To get the required isomorphism, we can simply use the identity on the cochain level. We also need to check that we get the same Steenrod square $\Sq^2 \colon H^k(\cC';\Z/2)\to H^{k+2}(\cC';\Z/2)$. If we consider the case $k=|y|$, and a cocycle is non-zero $y$, choosing the combinatorial matching to match $P$ and $M$ makes it clear that the same Steenrod square is obtained.

If $k=|y|-1$ we need to be more careful. Let $z\in H^k(\cC';\Z/2)=H^k(\cC;\Z/2)$ be represented by the cocycle $\varphi$. To determine $\sq^\varphi(x)$ we need to look at $\M(x,\varphi)$ which contains various intervals with endpoints $(B,M)$ and $(B,P)$. A combinatorial matching will match $B$ to some $B'$, resulting in further endpoints $(B',M)$ and $(B',P)$. In $\M(a',\varphi')$ the points $(B,M)$ and $(B,P)$ are identified, and so are $(B',M)$ and $(B',P)$. There are now various cases to consider, depending on the signs of $B$ and $B'$, but also on how the components in $\Gamma_\mathcal{C}(x,\varphi)$ and $\Gamma_\mathcal{C}(x',\varphi')$ are different.
We will consider a slightly more general case of this in the proof of Lemma \ref{lem:otherway} from which it will follow that the Steenrod square is indeed the same.
\end{proof}

\subsection{Combinatorial handle cancellation}
In view of \cite{JLS2} the next flow category move to transfer to $1$-flow categories is the handle cancellation. To save space we will only describe the most simple case of handle cancellation, but this case will be sufficient for our purposes later.

\begin{definition}
Let $(\cC,s,f)$ be a framed $1$-flow category, and let $x,y\in \Ob(\cC)$ satisfy $|x|=|y|+1$. Assume that $\M(x,y)$ consists of exactly one point, and that $\M(x,b)=\emptyset$ for all $b\in \Ob(\cC)-\{y\}$ with $|b|=|y|$, and $\M(a,y)=\emptyset$ for all $a\in \Ob(\cC)-\{x\}$ with $|a|=|x|$. Then let $(\cC',s',f')$ be the framed $1$-flow category given as follows:
\begin{itemize}
\item We have $\Ob(\cC')= \Ob(\cC)-\{x,y\}$ with $|\cdot |'\colon \Ob(\cC')\to \Z$ given by restriction of $|\cdot|$.
\item The moduli spaces $\M(a,b)$ of $\cC'$ agree with the moduli spaces of $\cC$.
\item Both $s'$ and $f'$ are the restrictions of $s$ and $f$, respectively.
\end{itemize}
\end{definition}

\begin{lemma}
The tuple $(\cC',s',f')$ is a framed $1$-flow category. Furthermore, there is an isomorphism 
\[
\Phi\colon H^\ast(\cC;\Z/2)\to H^\ast(\cC';\Z/2)
\]
which commutes with $\Sq^2$.
\end{lemma}

\begin{proof}
This is much simpler than the previous cases. We will only consider the compatibility condition for objects $a,b$ with $|a|=|x|+1$ and $|b|=|y|-1$.
Because $\cC$ is a framed $1$-flow category, we get $[a:x]=0$ and $[y:b]=0$ for all such $a,b\in \Ob(\cC)$. Furthermore, any intervals $I$ in $\M(a,y)$ and $J$ in $\M(x,b)$ will have both endpoints going through $x$, respectively $y$. Each such pair of intervals $(I,J)$ gives rise to an square $C$ in $\partial \M(a,b)$ with edges given by the intervals $\{+\}\times I$, $\{-\}\times I$, $J\times \{+\}$ and $J\times \{-\}$. Clearly $\tilde{f}(C)=1$, so $C$ does not contribute to the compatibility condition.

Note that $\partial \M(a,b)$ in $\cC'$ is obtained from $\partial \M(a,b)$ in $\cC$ by removing these squares. Since $\cC$ satisfies the compatibility condition, so does $\cC'$ now.
\end{proof}

\section{Computational aspects: the Smith normal form}
\label{sec:snf_for_flow_categories}

In Section \ref{sec:snf_for_flow_categories}, we finally describe the new algorithm for computing Steenrod squares in a framed flow category. We begin with the definition of a particularly simple type of 1-flow category, which the algorithm achieves upon termination.

\begin{definition}
 A framed $1$-flow category $\cC$ is said to be in \em primary Smith normal form\em, if the following are satisfied.
\begin{itemize}
 \item If $b$ is an object, there is at most one non-empty $0$-dimensional moduli space $\M(a,b)$, which contains $p^k$ elements, all of which have the same framing, where $p\geq 2$ is a prime and $k\geq 1$.
 \item For each object $a$ there exists at most one object $b$ with $|a|=|b|+1$ and $\M(a,b)$ non-empty.
\end{itemize}

\end{definition}

The idea of the algorithm is as follows: \begin{enumerate}
\item \label{item:idea1}restrict from a framed flow category to a framed 1-flow category, 
\item \label{item:idea2}use flow category moves to algorithmically reduce the 1-flow category to primary Smith normal form, 
\item \label{item:idea3}keep track of the relevant combinatorial data along the way so that the combinatorial Steenrod square can be used.
\end{enumerate}

Primary Smith normal form for flow categories was considered in \cite[\textsection 6]{ALPODS2}, where it was shown that flow category moves can be used to move any flow category into this simple form, so we already know that (\ref{item:idea2}) is always possible. As this section is computational in nature, the subtlety of our task is to come up with an algorithm to achieve (\ref{item:idea2}), which interacts with (\ref{item:idea3}) in a computationally efficient manner. For this, several tricks are used for reducing the amount of combinatorial data we have to carry with us.

\subsection{Ideas to speed up the process}

The Khovanov cohomology of a link has relatively few generators compared to the number of generators of the Khovanov cochain complex. Turning the Khovanov flow category of \cite{LipSarKhov} into a Smith normal form can be a time-consuming process. Furthermore, trying to keep track of the framed 1-dimensional moduli spaces is a difficult task as the numbers explode quickly. For example, for the diagram of the torus knot $T_{7,4}$ used to make Steenrod square calculations in \cite{JLS}, the number of points in the $0$-dimensional moduli spaces can exceed $10^{10000}$, while the largest torsion coefficient is $4$.

Large numbers do not have to represent a problem for a computer program, but points in the $0$-dimensional moduli spaces combine to endpoints of intervals in $1$-dimensional which can be combined in more diverse ways. These combinations may not behave well after applying handle slides.

We illustrate the key ideas for getting around this computational issue with an example.

\begin{example}
Assume a framed $1$-flow category $\cC$ contains objects $a,b_1,\ldots,b_5,c$ as in Figure \ref{fig:flowcat}. If we assume that every $0$-dimensional moduli space only contains points of the same sign, then there are 95 intervals in $\M(a,c)$. But we also need to know how their endpoints are paired. For example, of the 50 endpoints going through $b_1$, 20 may be paired with endpoints going through $b_3$, 17 with endpoints going through $b_4$ and the remaining 13 going through $b_5$.

Performing a handle slide where we slide $b_3$ three times over $b_1$ with subsequent Whitney trick requires us to keep track of how these numbers change.

On the other hand, if $c$ is part of a cocycle $\varphi$ whose Steenrod square we want to compute, we have to choose a combinatorial boundary matching $\mathcal{C}$ for $\varphi$, which will involve the objects $b_1,\ldots,b_5$. Clearly we can match the same four points in $\M(b_1,c)$ in the same way for every cocycle $\varphi$ involving $c$, while the fifth point needs to be matched with another point in a $\M(b_1,c')$, where $c'$ depends on $\varphi$.
Similarly in $\M(b_3,c)$ and $\M(b_4,c)$ we can even match all the points in the same way for every such $\varphi$. So when we do the actual Steenrod square calculation we add 75 edges to $\Gamma(a,c)$ in a way that will work for any $\varphi$ involving $c$. Let us call this new special graph structure $\tilde{\Gamma}(a,c)$, so that we have
\[
 \Gamma(a,c) \subset \tilde{\Gamma}(a,c) \subset \Gamma_\mathcal{C}(a,\varphi).
\]
Furthermore, several of the components $C$ in $\Gamma_\mathcal{C}(a,\varphi)$ may only involve edges already in $\tilde{\Gamma}(a,c)$. So their contribution to $\sq^\varphi(a)$ can be pre-computed and stored with $\M(a,c)$.

Notice that there are still 40 endpoints in $\tilde{\Gamma}(a,c)$, ten of which are of the form $(A_i,B_5)$ with $A_i\in \M(a,b_1)$ and $B_5\in \M(b_1,c)$. We can also choose a matching for the ten points in $\M(a,b_1)$ (and which would be relevant for calculating the Steenrod square of a cocycle involving $b_1$) and add five edges to $\tilde{\Gamma}(a,c)$ leading to a new special graph structure $\bar{\Gamma}(a,c)$. 
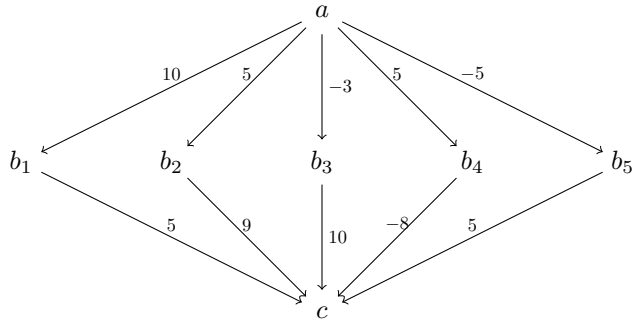
\begin{figure}[ht]
\begin{tikzpicture}
\node at (0,6) {$a$};
\node at (-4,4) {$b_1$};
\node at (-2,4) {$b_2$};
\node at (0,4) {$b_3$};
\node at (2,4) {$b_4$};
\node at (4,4) {$b_5$};
\node at (0,2) {$c$};

\draw [shorten >=0.3cm,shorten <=0.3cm,->] (0,6) -- node [above,scale=0.7] {$10$} (-4,4);
\draw [shorten >=0.3cm,shorten <=0.3cm,->] (0,6) -- node [above,scale=0.7] {$5$} (-2,4);
\draw [shorten >=0.3cm,shorten <=0.3cm,->] (0,6) -- node [right,scale=0.7] {$-3$} (0,4);
\draw [shorten >=0.3cm,shorten <=0.3cm,->] (0,6) -- node [above,scale=0.7] {$5$} (2,4);
\draw [shorten >=0.3cm,shorten <=0.3cm,->] (0,6) -- node [above,scale=0.7] {$-5$} (4,4);
\draw [shorten >=0.3cm,shorten <=0.3cm,->] (-4,4) -- node [above,scale=0.7] {$5$} (0,2);
\draw [shorten >=0.3cm,shorten <=0.3cm,->] (-2,4) -- node [above,scale=0.7] {$9$} (0,2);
\draw [shorten >=0.3cm,shorten <=0.3cm,->] (0,4) -- node [right,scale=0.7] {$10$} (0,2);
\draw [shorten >=0.3cm,shorten <=0.3cm,->] (2,4) -- node [above,scale=0.7] {$-8$} (0,2);
\draw [shorten >=0.3cm,shorten <=0.3cm,->] (4,4) -- node [above,scale=0.7] {$5$} (0,2);
\end{tikzpicture}
\caption{The $0$-dimensional moduli spaces between $a$ and $c$.\label{fig:flowcat}}
\end{figure}
After doing this for the objects $b_2$ and $b_5$ (we do not need to deal with $b_3$ and $b_4$ as there are no further endpoints left involving them), the resulting special graph structure $\bar{\Gamma}(a,c)$ has only two endpoints left, one going through $b_2$ and one going through $b_5$.

We would like to calculate the Steenrod square of $\varphi$ using the special graph structures $\bar{\Gamma}(a,c)$ by matching the remaining endpoints for varying $c$. An obvious problem is that $\bar{\Gamma}(a,c) \not\subset \Gamma_\mathcal{C}(a,c)$. This leads indeed to a difficulty, which requires us to define the \em distortion \em below. But it turns out that we can recover the Steenrod square from the $\bar{\Gamma}(a,c)$ directly, which will require much less information than the moduli spaces $\M(a,c)$.
\end{example}

\subsection{Partial combinatorial matchings}\label{subsec:partial}

For a framed 1-flow category $\cC$ with a combinatorial matching $\mathcal{C}$ for a cocycle $\varphi\in C^k(\cC;\Z/2)$ represented by the objects $c_1,\ldots, c_l$, we have the special graph structure $\Gamma_\mathcal{C}(a,\varphi)$ with assigned value $o(\Gamma_\mathcal{C}(a,\varphi))\in \Z/2$ from which we can derive the second Steenrod square. For any object $c$ (not necessarily the dual of a cocycle) we have $\Gamma(a,c)$ is defined for any object $a$ with $|a|=|c'|+2$. We would like to generalize $\Gamma_\mathcal{C}(a,c)$ to this non-cocycle setting, taking advantage of a global combinatorial matching.
Note that if for some $j$ we have that $\M(b_i,c_j)$ has an odd number of elements, there is one element that needs to be matched with an element of some $\M(b_i,c_{j'})$ where $j'\not=j$. We therefore cannot do a global combinatorial matching which would work for all possible Steenrod squares, as different cocycles require different generators from the flow category. But if $\M(b_i,c_j)$ has an even number of elements, we can match them all within, and use this matching whenever $c_j$ is part of the dual of a cocycle.

\begin{definition}
Let $\cC$ be a framed $1$-flow category. A \em partial combinatorial matching \em $\mathcal{P}$ is a collection of pairs $(A_1,A_2)$ where $A_1,A_2\in \M(a,b)$ are different points where $|a|=|b|+1$, and no point is paired up more than once.
\end{definition}

Notice that pairs are ordered, so if $A_1$ and $A_2$ have the same sign, we have a direction from $A_1$ to $A_2$ once we use this in a combinatorial matching.

The set of partial combinatorial matchings can be partially ordered using inclusion, and a maximal partial combinatorial matching has the property that if $\M(a,b)$ has an even number of points, all of them are paired up, and if there is an odd number of points in $\M(a,b)$, there is exactly one point which is not paired.

\begin{definition}
Let $\cC$ be a framed flow category and $\mathcal{P}$ a partial combinatorial matching. If $a,c\in \Ob(\cC)$ are objects with $|a|=|c|+2$, define a special graph structure $\Gamma_\mathcal{P}(a,c)$ as follows. The vertex set $V$ and $s\colon V \to \{0,1\}$ are the same as for $\Gamma(a,c)$, and the same holds for $E'$ and $f\colon E'\to \{0,1\}$. If $(B_1,A),(B_2,A) \in \M(b,c)\times \M(a,b)$, where $(B_1,B_2)\in \mathcal{P}$ we add an edge to $E-E'$ between the vertices $(B_1,A),(B_2,A)$ which is directed from $(B_1,A)$ to $(B_2,A)$ if $s(B_1)=s(B_2)$ and undirected otherwise.

If $(B,A_1),(B,A_2) \in \M(b,c)\times \M(a,b)$, where $(A_1,A_2)\in \mathcal{P}$ and $B$ is not matched to any other point $B'$ in $\mathcal{P}$, we add an edge to $E-E'$ between these vertices. It is undirected if $s(A_1)\not=s(A_2)$. If $s(A_1)=s(A_2)$ the edge is directed, pointing from $(B,A_1)$ to $(B,A_2)$. Furthermore, for every new edge between $(B,A_1)$ to $(B,A_2)$ (directed or not) we add a loop to $L$ if $s(B)=1\in \Z/2$.
\end{definition}

\begin{example}
Consider the $1$-flow category $\cC$ indicated in Figure \ref{fig:1}. The superscript in the elements of the $0$-dimensional moduli spaces indicate their sign.

\begin{figure}[ht]
\begin{tikzpicture}
\node at (0,6) {$a$};
\node at (-3,4) {$b_1$};
\node at (-1,4) {$b_2$};
\node at (1,4) {$b_3$};
\node at (3,4) {$b_4$};
\node at (0,2) {$c$};

\draw [shorten >=0.3cm,shorten <=0.3cm,->] (0,6) -- node [above,sloped,scale=0.7] {$B_1^+$} (-3,4);
\draw [shorten >=0.3cm,shorten <=0.3cm,->] (0,6) -- node [above,sloped,scale=0.7] {$B_2^-$} (-1,4);
\draw [shorten >=0.3cm,shorten <=0.3cm,->] (0,6) -- node [above,sloped,scale=0.7] {$B_3^-$} (1,4);
\draw [shorten >=0.3cm,shorten <=0.3cm,->] (0,6) -- node [above,sloped,scale=0.7] {$B_4^-$} (3,4);
\draw [shorten >=0.3cm,shorten <=0.3cm,->] (-3,4) -- node [above,sloped,scale=0.7] {$C_1^+$} (0,2);
\draw [shorten >=0.3cm,shorten <=0.3cm,->] (-1,4) -- node [above,sloped,scale=0.7] {$C_2^+,C_3^+$} (0,2);
\draw [shorten >=0.3cm,shorten <=0.3cm,->] (1,4) -- node [above,sloped,scale=0.7] {$C_2^-,C_3^-$} (0,2);
\draw [shorten >=0.3cm,shorten <=0.3cm,->] (3,4) -- node [above,sloped,scale=0.7] {$C_4^+$} (0,2);
\end{tikzpicture}
\caption{The flow category $\cC$.\label{fig:1}}
\end{figure}
The moduli space $\M(a,c)$ consists of three intervals, which we can choose to connect the points $(C_1^+,B_1^+)$ and $(C_4^+,B_4^-)$, $(C_2^+,B_2^-)$ and $(C_2^-,B_3^-)$, $(C_3^+,B_2^-)$ and $(C_3^-,B_3^-)$.

The partial matching pairs $(C_2^+,C_3^+)$ and $(C_2^-,C_3^-)$.  The special graph structure $\Gamma_\mathcal{P}(a,c)$ contains one interval with endpoints $(C_1^+,B_1^+)$ and $(C_4^+,B_4^-)$, and a circle consisting of four edges, two of which are oriented. Note that these two oriented edges point in opposite directions on the circle.

If we change the moduli space $\M(a,c)$ by choosing to connect the points $(C_1^+,B_1^+)$ and $(C_2^+,B_2^-)$, $(C_4^+,B_4^-)$ and $(C_2^-,B_3^-)$, $(C_3^+,B_2^-)$ and $(C_3^-,B_3^-)$, the graph will consist of one interval made up of five edges, two of which are oriented and again in opposite direction.

If we turn the $1$-flow category upside down, we get two extra loops in $\Gamma_\mathcal{P}(c^\ast,a^\ast)$ coming from the fact that $B_2^-$ and $B_3^-$ are negatively framed (one loop from the edge between $(B_2^-,C_2^+)$ and $(B_2^-,C_3^+)$, and one from the edge between $(B_3^-,C_2^-)$ and $(B_3^-,C_3^-)$).

Note however that in this example $c$ does not represent a cocycle, and $a$ represents a coboundary, so this flow category does not contain interesting Steenrod squares.
\end{example}

\subsection{Steenrod squares using partial combinatorial matchings}
\label{subsec:steen_sq_and_partial_comb_matchings}

\begin{definition}For $\Gamma$ a special graph structure and $C$ a component of $\Gamma$ which does not contain any vertices of $\partial V$ define $o(C)\in\Z/2$ as follows. If $C$ is in $L$ then $o(C)=1$. If $C$ contains a vertex then $o(C):=1+F(C)+D(C)$.
\end{definition}

With this new notation, if for some $\Gamma$ we have $\partial V=\emptyset$ then we may write $o(\Gamma)=\sum_C o(C)$ where the sum is over all components of $\Gamma$.

Now consider again an arbitrary framed $1$-flow category $\cC$.
If a cocycle $\varphi$ is represented by $c_1,\ldots, c_l$ we want to obtain values $\sq^\varphi(a)\in \Z/2$ for objects $a\in \Ob(\cC)$ using $\Gamma_\mathcal{P}(a,c_j)$ giving rise to the Steenrod square. Observe that the special graph structure $\Gamma_{\mathcal{P}}(a,c_j)$ has as boundary points the pairs $(B,A)$ with neither $B$ nor $A$ matched by $\mathcal{P}$.

Denote by $\Gamma_{\mathcal{P}}(a,\varphi)$ the special graph structure given by the disjoint union of the $\Gamma_\mathcal{P}(a,c_j)$. If we extend $\mathcal{P}$ to a combinatorial matching $\mathcal{C}$ for $\varphi$ by pairing those $B\in \mathcal{M}(b_i,c_j)$ which are not paired in $\mathcal{P}$, we can extend $\Gamma_{\mathcal{P}}(a,\varphi)$ to a special graph structure $\overline{\Gamma}_\mathcal{C}(a,\varphi)$ by adding edges between vertices $(B,A)$ and $(B',A)$, where $(B,B')$ are matched in $\mathcal{C}$, but $B,B'$ and $A$ are not matched in $\mathcal{P}$.

Notice that $\overline{\Gamma}_\mathcal{C}(a,\varphi)$ differs from $\Gamma_\mathcal{C}(a,\varphi)$ in that there are edges in the former between vertices $(B,A)$ and $(B,A')$ if $B$ is not matched in $\mathcal{P}$, but $A$ and $A'$ are. Nevertheless, $\overline{\Gamma}_\mathcal{C}(a,\varphi)$ has no boundary points, so we can assign each component $C$ a value $o(C)$. Adding these values does not give the right value for the Steenrod square, so we need to make a further definition first.

\begin{definition}\label{def:specgraphstructure}
Given a cocycle $\varphi$, with respect to $\Z/2$-coefficients, represented by objects $c_1,\ldots,c_l$, and an object $a\in \Ob(\cC)$ with $|a|=|c_j|+2$, let $b_1,\ldots,b_m\in \Ob(\cC)$ be the objects with $\M(b_i,c_j)\times \M(a,b_i)$ non-empty for at least one $c_j$, and let $\mathcal{P}$ be a maximal partial combinatorial matching.

Let $u_i$ be the number of moduli spaces $\M(b_i,c_j)$ for $j=1,\ldots,l$ which have an odd number of elements. Note that $u_i$ is even as $\varphi$ is a cocycle. 

Also, let $v_i$ be the number of pairings in $\mathcal{P}$ involving elements of $\M(a,b_i)$ between points of the same sign. 

Define the \em distortion of $\mathcal{P}$ \em by \[d_\mathcal{P}(a,\varphi)=\sum_i (u_i/2)v_i \in \Z/2.\]

Furthermore, define
\[
 \widetilde{\sq^\varphi}(a) = \sum_C o(C) + d_\mathcal{P}(a,\varphi) \in \Z/2
\]
where the sum is taken over the components of $\overline{\Gamma}_\mathcal{C}(a,\varphi)$ for some extension $\mathcal{C}$ of $\mathcal{P}$.
\end{definition}

\begin{lemma}\label{lem:otherway}
Let $\varphi\in C^k(\cC;\Z/2)$ be a cocycle and $a\in \Ob(\cC)$ with $|a|=k+2$. Let $\mathcal{P}$ be a maximal partial combinatorial matching and $\mathcal{C}$ an extension of $\mathcal{P}$ to a combinatorial matching for $\varphi$. Then $\sq^\varphi(a)=\widetilde{\sq^\varphi}(a)\in \Z/2$.
\end{lemma}

\begin{proof}
Let $B\in \mathcal{M}(b_i,c_j)$ and $B'\in \mathcal{M}(b_i,c_k)$ be matched in $\mathcal{C}$, but not in $\mathcal{P}$. Also, let $A,A'\in \mathcal{M}(a,b_i)$ be matched in $\mathcal{P}$. The special graph structures $\overline{\Gamma}_\mathcal{C}(a,c)$ and $\Gamma_\mathcal{C}(a,\varphi)$ differ in that in the former there is an edge between $(B,A)$ and $(B,A')$, and one between $(B',A)$ and $(B',A')$, while in the latter there is an edge between $(B,A)$ and $(B',A)$, and one between $(B,A')$ and $(B',A')$. Also, if $s(B)=1$, there will be an extra loop in $L$ coming with every edge $(B,A)$ to $(B,A')$.

We need to see how the sum over the components $\sum_C o(C)$ changes when we change these edges. We have to check four cases, depending on whether $s(B)=s(B')$ and $s(A)=s(A')$ or not. Let us begin with the case $s(A)\not=s(A')$ and $s(B)\not= s(B')$. In this case all four relevant edges are non-directed. We basically have the following picture
\begin{center}
\begin{tikzpicture}
\node[circle,draw, inner sep=0pt, minimum size=3pt] at (1,0) {$ $};
\node[circle,draw, inner sep=0pt, minimum size=3pt] at (2,0) {$ $};
\node[circle,draw, inner sep=0pt, minimum size=3pt] at (3,0) {$ $};
\node[circle,draw, inner sep=0pt, minimum size=3pt] at (4,0) {$ $};

\node[circle,draw, inner sep=0pt, minimum size=3pt] at (1,1) {$ $};
\node[circle,draw, inner sep=0pt, minimum size=3pt] at (2,1) {$ $};
\node[circle,draw, inner sep=0pt, minimum size=3pt] at (3,1) {$ $};
\node[circle,draw, inner sep=0pt, minimum size=3pt] at (4,1) {$ $};

\draw[-] (3.05,0) -- (3.95,0) node [above,pos=0.5] {$0$};
\draw[-] (1.05,0) -- (1.95,0) node [above,pos=0.5] {$0$};
\draw[-] (2.05,0) -- (2.95,0);
\draw[-] (3.05,1) -- (3.95,1) node [above,pos=0.5] {$0$};
\draw[-] (1.05,1) -- (1.95,1) node [above,pos=0.5] {$0$};
\draw[-] (2.05,1) -- (2.95,1);

\node [scale=0.7] at (2,-0.3) {$(B',A)$};
\node [scale=0.7] at (3,-0.3) {$(B',A')$};
\node [scale=0.7] at (2,1.2) {$(B,A)$};
\node [scale=0.7] at (3,1.2) {$(B,A')$};
\node [scale=0.7] at (0.7,0) {$Z$};
\node [scale=0.7] at (4.3,0) {$W$};
\node [scale=0.7] at (0.7,1) {$X$};
\node [scale=0.7] at (4.3,1) {$Y$};

\node[circle,draw, inner sep=0pt, minimum size=3pt] at (6,0) {$ $};
\node[circle,draw, inner sep=0pt, minimum size=3pt] at (7,0) {$ $};
\node[circle,draw, inner sep=0pt, minimum size=3pt] at (8,0) {$ $};
\node[circle,draw, inner sep=0pt, minimum size=3pt] at (9,0) {$ $};

\node[circle,draw, inner sep=0pt, minimum size=3pt] at (6,1) {$ $};
\node[circle,draw, inner sep=0pt, minimum size=3pt] at (7,1) {$ $};
\node[circle,draw, inner sep=0pt, minimum size=3pt] at (8,1) {$ $};
\node[circle,draw, inner sep=0pt, minimum size=3pt] at (9,1) {$ $};

\draw[-] (8.05,0) -- (8.95,0) node [above,pos=0.5] {$0$};
\draw[-] (6.05,0) -- (6.95,0) node [above,pos=0.5] {$0$};
\draw[-] (7,0.05) -- (7,0.95);
\draw[-] (8.05,1) -- (8.95,1) node [above,pos=0.5] {$0$};
\draw[-] (6.05,1) -- (6.95,1) node [above,pos=0.5] {$0$};
\draw[-] (8,0.05) -- (8,0.95);

\node [scale=0.7] at (7,-0.3) {$(B',A)$};
\node [scale=0.7] at (8,-0.3) {$(B',A')$};
\node [scale=0.7] at (7,1.2) {$(B,A)$};
\node [scale=0.7] at (8,1.2) {$(B,A')$};
\node [scale=0.7] at (5.7,0) {$Z$};
\node [scale=0.7] at (9.3,0) {$W$};
\node [scale=0.7] at (5.7,1) {$X$};
\node [scale=0.7] at (9.3,1) {$Y$};
\end{tikzpicture}
\end{center}
with the left-hand side a subgraph of $\overline{\Gamma}_\mathcal{C}(a,\varphi)$ and the right-hand side a subgraph of $\Gamma_\mathcal{C}(a,\varphi)$. Note that we can assume all marked edges to be marked by $0$ because of the local move (c).

In $\overline{\Gamma}_\mathcal{C}(a,\varphi)$ the endpoints $X,Y,Z$ and $W$ are connected somehow, and there are three ways of doing this. First, $X$ and $Y$ connect, and $Z$ and $W$ connect. This leads to the picture
\begin{center}
\begin{tikzpicture}
\node[circle,draw, inner sep=0pt, minimum size=3pt] at (1,0) {$ $};
\node[circle,draw, inner sep=0pt, minimum size=3pt] at (2,0) {$ $};
\node[circle,draw, inner sep=0pt, minimum size=3pt] at (3,0) {$ $};
\node[circle,draw, inner sep=0pt, minimum size=3pt] at (4,0) {$ $};

\node[circle,draw, inner sep=0pt, minimum size=3pt] at (1,1) {$ $};
\node[circle,draw, inner sep=0pt, minimum size=3pt] at (2,1) {$ $};
\node[circle,draw, inner sep=0pt, minimum size=3pt] at (3,1) {$ $};
\node[circle,draw, inner sep=0pt, minimum size=3pt] at (4,1) {$ $};

\draw[-] (3.05,0) -- (3.95,0);
\draw[-] (1.05,0) -- (1.95,0);
\draw[-] (2.05,0) -- (2.95,0);
\draw[-] (3.05,1) -- (3.95,1);
\draw[-] (1.05,1) -- (1.95,1);
\draw[-] (2.05,1) -- (2.95,1);

\node[circle,draw, inner sep=0pt, minimum size=3pt] at (6,0) {$ $};
\node[circle,draw, inner sep=0pt, minimum size=3pt] at (7,0) {$ $};
\node[circle,draw, inner sep=0pt, minimum size=3pt] at (8,0) {$ $};
\node[circle,draw, inner sep=0pt, minimum size=3pt] at (9,0) {$ $};

\node[circle,draw, inner sep=0pt, minimum size=3pt] at (6,1) {$ $};
\node[circle,draw, inner sep=0pt, minimum size=3pt] at (7,1) {$ $};
\node[circle,draw, inner sep=0pt, minimum size=3pt] at (8,1) {$ $};
\node[circle,draw, inner sep=0pt, minimum size=3pt] at (9,1) {$ $};

\draw[-] (8.05,0) -- (8.95,0);
\draw[-] (6.05,0) -- (6.95,0);
\draw[-] (7,0.05) -- (7,0.95);
\draw[-] (8.05,1) -- (8.95,1);
\draw[-] (6.05,1) -- (6.95,1);
\draw[-] (8,0.05) -- (8,0.95);

\draw[-, bend right] (1.05,0) to (3.95,0);
\draw[-, bend left] (1.05,1) to (3.95,1);
\draw[-, bend right] (6.05,0) to (8.95,0);
\draw[-, bend left] (6.05,1) to (8.95,1);
\end{tikzpicture}
\end{center}
If we assume that there are $a$ directed edges on the outside path from $Y$ to $X$ pointing in a fixed direction on the component, and $b$ directed edges on the outside path from $Z$ to $W$ pointing in a given direction, the contribution to the sum $o(C)$ over the components from these two components is $2+1+a+b$, where $2$ comes from the two components, the $1$ comes from an extra loop in $L$, since $s(B)\not=s(B')$ means that one of those signs is $1$, and so one of the edges comes with an extra loop. The summands $a$ and $b$ come from $D(C)$.

On the right we only get one component $C$ and there will be no extra loop, so the contribution is $1+D(C)$. But it is easy to see that $D(C)=a+b$, so both pictures contribute the same to $\sum_Co(C)\in \Z/2$.

The second case connects $X$ with $Z$ and $Y$ with $W$, which is basically the same case as the previous one.

In the third case we connect $X$ with $W$ and $Y$ with $Z$. We now get one component in the graph in both cases, but if we choose the orientation to be along the path from $X$ to $W$, the orientation on the path from $Y$ to $Z$ changes.
\begin{center}
\begin{tikzpicture}
\node[circle,draw, inner sep=0pt, minimum size=3pt] at (1,0) {$ $};
\node[circle,draw, inner sep=0pt, minimum size=3pt] at (2,0) {$ $};
\node[circle,draw, inner sep=0pt, minimum size=3pt] at (3,0) {$ $};
\node[circle,draw, inner sep=0pt, minimum size=3pt] at (4,0) {$ $};

\node[circle,draw, inner sep=0pt, minimum size=3pt] at (1,1) {$ $};
\node[circle,draw, inner sep=0pt, minimum size=3pt] at (2,1) {$ $};
\node[circle,draw, inner sep=0pt, minimum size=3pt] at (3,1) {$ $};
\node[circle,draw, inner sep=0pt, minimum size=3pt] at (4,1) {$ $};

\draw[-] (3.05,0) -- (3.95,0);
\draw[-] (1.05,0) -- (1.95,0);
\draw[-] (2.05,0) -- (2.95,0);
\draw[-] (3.05,1) -- (3.95,1);
\draw[-] (1.05,1) -- (1.95,1);
\draw[-] (2.05,1) -- (2.95,1);

\node[circle,draw, inner sep=0pt, minimum size=3pt] at (6,0) {$ $};
\node[circle,draw, inner sep=0pt, minimum size=3pt] at (7,0) {$ $};
\node[circle,draw, inner sep=0pt, minimum size=3pt] at (8,0) {$ $};
\node[circle,draw, inner sep=0pt, minimum size=3pt] at (9,0) {$ $};

\node[circle,draw, inner sep=0pt, minimum size=3pt] at (6,1) {$ $};
\node[circle,draw, inner sep=0pt, minimum size=3pt] at (7,1) {$ $};
\node[circle,draw, inner sep=0pt, minimum size=3pt] at (8,1) {$ $};
\node[circle,draw, inner sep=0pt, minimum size=3pt] at (9,1) {$ $};

\draw[-] (8.05,0) -- (8.95,0);
\draw[-] (6.05,0) -- (6.95,0);
\draw[-] (7,0.05) -- (7,0.95);
\draw[-] (8.05,1) -- (8.95,1);
\draw[-] (6.05,1) -- (6.95,1);
\draw[-] (8,0.05) -- (8,0.95);

\draw[->] (1.05,0) -- (3.95,1);
\draw[->] (1.05,1) -- (3.95,0);
\draw[->] (6.05,0) -- (8.95,1);
\draw[<-] (6.05,1) -- (8.95,0);
\end{tikzpicture}
\end{center}
Note that the assumption $s(A)\not=s(A')$ and $s(B)\not= s(B')$ implies that $s(Y)=s(W)$. Hence if there are $a$ directed edges between $Y$ and $W$ pointing from $Y$ to $W$, then there are $1+a$ directed edges between $Y$ and $W$ pointing from $W$ to $Y$. The contributions to $o(C)$ thus differ by $1$, but since on the left we would get one extra loop, the total contribution to $\sum_Co(C)$ is the same again.

Now assume that $s(A)\not=s(A')$ and $s(B)=s(B')$. The differences to the previous case are that there are no extra loops (or two extra loops, which we can ignore) from $s(B)=1$ or not, and that the vertical edges on the right are directed, which we can assume to both point downwards.

Again we have three cases to check, and it turns out that again all contributions to $\sum_Co(C)$ are the same. For example, in the third case we have to count either both of the two extra directed edges, or none of them. As we count mod $2$, this has no effect.

Now consider the assumption $s(A)=s(A')$ and $s(B)\not=s(B')$, leading again to an extra loop in $\overline{\Gamma}_\mathcal{C}(a,c)$. This leads to the following picture.
\begin{center}
\begin{tikzpicture}
\node[circle,draw, inner sep=0pt, minimum size=3pt] at (1,0) {$ $};
\node[circle,draw, inner sep=0pt, minimum size=3pt] at (2,0) {$ $};
\node[circle,draw, inner sep=0pt, minimum size=3pt] at (3,0) {$ $};
\node[circle,draw, inner sep=0pt, minimum size=3pt] at (4,0) {$ $};

\node[circle,draw, inner sep=0pt, minimum size=3pt] at (1,1) {$ $};
\node[circle,draw, inner sep=0pt, minimum size=3pt] at (2,1) {$ $};
\node[circle,draw, inner sep=0pt, minimum size=3pt] at (3,1) {$ $};
\node[circle,draw, inner sep=0pt, minimum size=3pt] at (4,1) {$ $};

\draw[-] (3.05,0) -- (3.95,0) node [above,pos=0.5] {$0$};
\draw[-] (1.05,0) -- (1.95,0) node [above,pos=0.5] {$0$};
\draw[-] (2.05,0) -- (2.95,0);
\draw[-] (3.05,1) -- (3.95,1) node [above,pos=0.5] {$0$};
\draw[-] (1.05,1) -- (1.95,1) node [above,pos=0.5] {$0$};
\draw[-] (2.05,1) -- (2.95,1);

\node [scale=0.7] at (2,-0.3) {$(B',A)$};
\node [scale=0.7] at (3,-0.3) {$(B',A')$};
\node [scale=0.7] at (2,1.2) {$(B,A)$};
\node [scale=0.7] at (3,1.2) {$(B,A')$};
\node [scale=0.7] at (0.7,0) {$Z$};
\node [scale=0.7] at (4.3,0) {$W$};
\node [scale=0.7] at (0.7,1) {$X$};
\node [scale=0.7] at (4.3,1) {$Y$};

\node[circle,draw, inner sep=0pt, minimum size=3pt] at (6,0) {$ $};
\node[circle,draw, inner sep=0pt, minimum size=3pt] at (7,0) {$ $};
\node[circle,draw, inner sep=0pt, minimum size=3pt] at (8,0) {$ $};
\node[circle,draw, inner sep=0pt, minimum size=3pt] at (9,0) {$ $};

\node[circle,draw, inner sep=0pt, minimum size=3pt] at (6,1) {$ $};
\node[circle,draw, inner sep=0pt, minimum size=3pt] at (7,1) {$ $};
\node[circle,draw, inner sep=0pt, minimum size=3pt] at (8,1) {$ $};
\node[circle,draw, inner sep=0pt, minimum size=3pt] at (9,1) {$ $};

\draw[-] (8.05,0) -- (8.95,0) node [above,pos=0.5] {$0$};
\draw[-] (6.05,0) -- (6.95,0) node [above,pos=0.5] {$0$};
\draw[-] (7,0.05) -- (7,0.95);
\draw[-] (8.05,1) -- (8.95,1) node [above,pos=0.5] {$0$};
\draw[-] (6.05,1) -- (6.95,1) node [above,pos=0.5] {$0$};
\draw[-] (8,0.05) -- (8,0.95);

\node [scale=0.7] at (7,-0.3) {$(B',A)$};
\node [scale=0.7] at (8,-0.3) {$(B',A')$};
\node [scale=0.7] at (7,1.2) {$(B,A)$};
\node [scale=0.7] at (8,1.2) {$(B,A')$};
\node [scale=0.7] at (5.7,0) {$Z$};
\node [scale=0.7] at (9.3,0) {$W$};
\node [scale=0.7] at (5.7,1) {$X$};
\node [scale=0.7] at (9.3,1) {$Y$};

\draw[->] (2.2,1) -- (2.5,1);
\draw[->] (2.2,0) -- (2.5,0);
\end{tikzpicture}
\end{center}
An argument as above shows that for all three cases of connecting the endpoints $X,Y,Z$ and $W$ the contribution to $\sum_Co(C)$ this time around differs by $1$. However, notice that since $s(A)=s(A')$ this pair contributes to $v_i$ for the appropriate $i$. Also, since $\mathcal{P}$ is maximal and $B,B'$ are not matched by $\mathcal{P}$, this pair contributes $1$ to $u_i/2$. Thus, taking the distortion into account will lead to the correct result.

The final case $s(A)=s(A')$ and $s(B)= s(B')$ is again similar, requiring again the distortion to give the right outcome.
\end{proof}

\begin{remark}
Recall that we did not quite finish the proof that the Whitney trick does not affect the Steenrod square in the case where $k=|y|-1$. In view of partial combinatorial matchings, we can think of $\sq^c(x')$ as being obtained using a partial matching which combines $(B,P)$ with $(B,M)$ and $(B',P)$ with $(B',M)$ instead of $(B,P)$ with $(B',P)$ and $(B,M)$ with $(B',M)$, so the above proof shows that the Steenrod square is the same for both ways. Notice that distortion plays no role as $s(P)\not=s(M)$.
\end{remark}

Lemma \ref{lem:otherway} states that we can calculate the Steenrod square from the information given by $\Gamma_\mathcal{P}(a,\varphi)$, once we extend $\mathcal{P}$ to a combinatorial matching $\mathcal{C}$ for $\varphi$. If the $1$-flow category is in primary Smith normal form, the values $u_i$ are all $0$. In the process of turning a $1$-flow category into primary Smith normal form we need to keep track of the changes in distortions.

Instead of keeping track of the $1$-dimensional moduli spaces $\M(a,c)$ to calculate the Steenrod square, we can now just keep track of the $\Gamma_\mathcal{P}(a,c)$. From the point of view of a computer program, this is presumably not much of an improvement, but in fact we only need to keep track of the equivalence class of each $\Gamma_\mathcal{P}(a,c)$. This is going to be an improvement, as we can always find a rather simple representative of the special graph structure.

Recall that by definition, vertices in $\Gamma$ have valency at most 2, so the connected components of $\Gamma$ consist of intervals (with vertices) or circles.

\begin{lemma}\label{lem:nicegamma}
Let $\Gamma$ be a special graph structure. Then $\Gamma$ is equivalent to a special graph structure $\Gamma'$ which has at most one circle component and in which case this component is a loop in $L$. Furthermore, the interval components of $\Gamma'$ either consist of a single edge $e\in E'$ with $f(e)=0$, or of three edges $e_1,e_2,e$ with $e_1,e_2\in E'$, $e\in E''$, and $f(e_1)=0=f(e_2)$.
\end{lemma}

\begin{proof}
The proof is by induction on the number of edges in $\Gamma$. If $\Gamma$ has only one edge, this edge $e$ will form an interval component. If $f(e)=1$, we can use a move (c) to change the framing by adding a loop. If the set of loops $L$ consists of more than one element, use moves (e) to reduce the number of loops. This finishes the induction start.

Now assume the result holds for all special graph structures with no more than $n$ edges, $n\geq 1$, and let $\Gamma$ be a special graph structure wih $n+1$ edges. If every edge is labelled with $0$ or $1$, each edge is its own interval component, and we can use moves (c) and (e) to bring the special graph structure into the desired form. Otherwise the set of unlabelled edges $E-E'$ is non-empty. Notice that the endpoints of $e\in E-E'$ are not in $\partial V$.

Assume there exists an undirected edge $e\in E-E'$. Then, up to (c) moves we can assume that $e$ fits into the left-hand side of a move (g) or (h). In both cases, we can reduce the number of edges up to equivalence and invoke induction.

If all the edges in $E-E'$ are directed, there either exists a component which has more than one directed edge, or all the components are interval components with three edges. In the latter case we can use moves (c) to get all labels to be $0$, and then reduce the number of loops with move (e) to get the desired form. If there exists a component with more than one directed edge, first use moves (d) to ensure that all directed edges in this component have the same direction. As there are no undirected edges by assumption, the left hand-side of move (f) fits into the component (after possibly another move (c)). Up to equivalence, we can again reduce the number of edges, and hence use induction. The result follows.
\end{proof}

\begin{definition}
Let $(\Gamma, I)$ consist of $\Gamma$ a special graph structure and $I$ an ordered pairing of the boundary points $\partial V$ of $\Gamma$. Let $\Gamma'$ be a special graph structure equivalent to $\Gamma$ as in Lemma \ref{lem:nicegamma}. Then $I$ is also an ordered pairing of the boundary points of $\Gamma'$. Furthermore, assume that the boundary points of every interval component of $\Gamma'$ are paired in $I$, and that if the interval component contains a directed edge then the direction of the edge agrees with the direction coming from the ordered pairing in $I$.

Given such a $\Gamma'$, define
\[
 o(\Gamma,I) = o(\Gamma''),
\]
where $\Gamma''$ is the sub-special graph structure of $\Gamma'$ where all interval components have been removed.
\end{definition}

\begin{lemma}\label{lem:whitneyeasy}
Let $(\cC,s,f)$ be a framed $1$-flow category and $\mathcal{P}$ a partial combinatorial matching. Let $x,y\in \Ob(\cC)$ with $|x|=|y|+1$ and assume $\M(x,y)$ contains two points $P,M$ with $s(P)=0$ and $s(M)=1$. Assume further that either $P,M$ are matched in $\mathcal{P}$ or there exist $P',M'\in \M(x,y)$ with $s(P')=0$ and $s(M')=1$ such that $(P,P'), (M,M')\in\mathcal{P}$.

Let $(\cC',s',f')$ be the framed $1$-flow category obtained by performing the Whitney trick on $P,M$, and let $\mathcal{P}'$ be the partial combinatorial matching obtained by removing the pairings involving $P,M$ and adding $(P',M')$ if $P$ and $M$ are not paired in $\mathcal{P}$. Then
\[
 o(\Gamma_\mathcal{P}(a,y),I) = o(\Gamma_{\mathcal{P}'}(a',y'),I)
\]
and
\[
 o(\Gamma_\mathcal{P}(x,d),J) = o(\Gamma_{\mathcal{P}'}(x',d'),J)
\]
for all $a,d\in \Ob(\cC)$ with $|a|=|x|+1$ and $|d|=|y|-1$, all $I$ ordered pairings of boundary points of $\Gamma_\mathcal{P}(a,y)$, and all $J$ ordered pairings of boundary points of $\Gamma_\mathcal{P}(x,d)$.
\end{lemma}

\begin{proof}
If $P,M$ are matched, the result for $(a,y)$ follows directly from move (g). For $(x,d)$ consider the vertices $(B,P), (B,M)\in \M(c,d)\times \M(x,c)$ for some object $c$. If $B$ is not matched in $\mathcal{P}$ and we get an edge between $(B,P)$ and $(B,M)$ in $\Gamma_\mathcal{P}(x,d)$ together with a loop if $s(B)=1$.  Using move (g) we get exactly $\Gamma_\mathcal{P}(x',d')$ by the definition of the Whitney trick. If $B$ is matched to some $B'$, we get from $\Gamma_\mathcal{P}(x,d)$ to $\Gamma_\mathcal{P}(x',d')$ with a move (a).

Now assume $P$ is matched to $P'$ and $M$ is matched to $M'$. For $(a,y)$ we have to use a move (a) to get the desired result, while for $(x,d)$ we have to consider again vertices $(B,P)$ and $(B,M)$ together with the cases whether $B$ is matched or not. Again one checks that the special graph structures are equivalent.
\end{proof}

\begin{remark}
Lemma \ref{lem:whitneyeasy} basically tells us that we can perform Whitney tricks without having to worry too much about partial combinatorial matchings. For a computer program the Whitney trick is more or less automatic.
\end{remark}

\subsection{Effects of handle slides on special graph structures}\label{subsec:effects}
Handle slides affect the special graph structures in a more significant way. The following lemmas describe how the various special graph structures change during the process of turning the 1-flow category into Smith normal form.

Let $\cC$ be a framed $1$-flow category and $\mathcal{P}$ a partial combinatorial matching, where we assume that all points of a $0$-dimensional moduli space $\M(a,b)$ have the same sign (this can be achieved using the Whitney trick). Let $\cC^\varepsilon_S$ be the result of sliding an object $x$ over $y$. We can define a partial combinatorial matching $\mathcal{P'}$ for $\cC^\varepsilon_S$ as follows. The inclusion $\M(a,b)\subset \M(a',b')$ for all objects $a,b$ with $|a|=|b|+1$ allows us to consider $\mathcal{P}\subset \mathcal{P'}$.
If $Y\in \M(y,b)$, write $Y'\in \M(x',b')$ for the new point, and similarly write $X'\in \M(a',y')$ for $X\in \M(a,x)$. If $Y_1,Y_2\in \M(y,b)$ are paired as $(Y_1,Y_2)\in \mathcal{P}$, we also pair $(Y_1',Y_2')\in \mathcal{P'}$. For $X_1,X_2\in \M(a,x)$ paired as $(X_1,X_2)\in \mathcal{P}$, we pair them as $(X_1',X_2')\in \mathcal{P'}$ provided that $\varepsilon = 0$. If $\varepsilon=1$ we pair them as $(X_2',X_1')\in \mathcal{P}'$.
Note that we assume that $X_1$ and $X_2$ have the same sign, so the order matters. The assumption that they have the same sign is justified by the following.

\begin{lemma}\label{lem:singleslide}
Let $\cC$ be a framed $1$-flow category with all points of a $0$-dimensional moduli space having the same sign, and $\mathcal{P}$ a partial combinatorial matching. Let $\cC'$ be the result of sliding an object $x$ over $y$, and let  $\mathcal{P}'$ be the partial combinatorial matching for $\cC'$ described above. Then we have
\[
 \Gamma_{\mathcal{P}'}(x',d') = \Gamma_\mathcal{P}(x,d) \sqcup \Gamma'_\mathcal{P}(y,d)
\]
where $\Gamma'_\mathcal{P}(y,d)$ agrees with $\Gamma_\mathcal{P}(y,d)$ except that the sign function $s'$ on vertices is given by $s'(X,Y)= \varepsilon + s(X,Y)$ for the vertices $(X,Y)\in \M(c',d')\times \M(y',c')$.

Furthermore,
\[
 \Gamma_{\mathcal{P}'}(b',d') = \Gamma_\mathcal{P}(b,d)
\]
for all $b\in \Ob(\cC)$ with $|b|=|x|$, $b\not=x$ and $|d|=|x|-2$, and
\[
 \Gamma_{\mathcal{P}'}(a',c') = \Gamma_{\mathcal{P}}(a,c) \sqcup I_\mathcal{P}(a,c),
\]
where $a,c\in \Ob(\cC)$ with $|a|=|x|+1=|c|+2$, and $I_\mathcal{P}(a,c)$ is a special graph structure containing one interval $I_{B,A}$ for every pair of points $(B,A)\in \M(y,c)\times \M(a,x)$ with neither $B$ nor $A$ matched by $\mathcal{P}$. The framing value of $I_{(B,A)}$ is given by $\varepsilon+\varepsilon_B$.
\end{lemma}

\begin{proof}
We have $\M(x',d')=\M(x,d)\sqcup \M(y,d)$ and the signs on vertices $(X,Y)\in \M(c',d')\times \M(y',c')$ are given by $s'(X,Y)=\varepsilon+s(X,Y)$ with the $\varepsilon$ coming from the new sign of $Y\in \M(y',c')$. The framings of intervals in $\M(y,d)$ do not change. Adding edges from the partial combinatorial matching do not change directions or add extra loops as the signs of $\M(c',d')$ do not change, and signs of $\M(y',c')$ change uniformly by $\varepsilon$. Hence $\Gamma_{\mathcal{P}'}(x',d') = \Gamma_\mathcal{P}(x,d) \sqcup \Gamma'_\mathcal{P}(y,d)$ as claimed.

We also get $\Gamma_{\mathcal{P}'}(b',d') = \Gamma_\mathcal{P}(b,d)$ for $b\not=x$ with $|b|=|x|$ as there is no change for the special graph structure.

Now for $a,c$ with $|a|=|x|+1=|c|+2$ we have
\[
 \M(a',c') = \M(a,c) \sqcup \M(y,c)\times [0,1] \times \M(a,x)
\]
which means we get a new interval for every pair $(B,A)\in \M(y,c) \times \M(a,x)$ with endpoints $(B',A)$ and $(B,A')$, where $A'\in \M(a',y')$ and $B'\in \M(x',c')$ are new elements.

If $B$ is matched to $\bar{B}$ in $\mathcal{P}$, then $B'$ is matched to $\bar{B}'$ in $\mathcal{P}'$, and the intervals corresponding to $(B,A)$ and $(\bar{B},A)$ form a circle component $C$ together with two oriented edges coming from the partial combinatorial matching. Since $s(B)=s(\bar{B})$ by assumption, and $B'$ and $\bar{B}'$ are matched in the same order as $B$ and $\bar{B}'$, we get $o(C) = 1 + 1 = 0$, so we can drop $C$ from $\Gamma_{\mathcal{P}'}(a',c')$.

Similarly, if $B$ is not matched in $\mathcal{P}$, but $A$ is matched to $\bar{A}\in \M(a,x)$, the intervals corresponding to $(B,A)$ and $(B,\bar{A})$ form a circle component $C$ together with two oriented edges from $\mathcal{P}'$. Notice that these oriented edges come with an extra loop if $s'(B)$ or $s'(B')$ equal $1$. As $s'(B)+s'(B')=\varepsilon$, we get an extra loop if $\varepsilon=1$. But notice that for $\varepsilon=1$ the pair $A'$ and $\bar{A}'$ are directed opposite to the pair $A$ and $\bar{A}$, so that this extra loop is cancelled with $o(C)$. So again this has no impact on $\Gamma_{\mathcal{P}'}(a',c')$.
\end{proof}

Notice that even if $\mathcal{P}$ is maximal, $\mathcal{P'}$ as described above need not be maximal. But if $\mathcal{P}$ is maximal, each $\M(x,b)$ contains at most one point which is not matched. So if $\M(x,c)$ and $\M(y,c)$ each have one point which is not matched in $\mathcal{P}$, these two points are not matched in $\mathcal{P'}$ although they both are in $\M(x',c')$. Let $\mathcal{P''}$ be the partial combinatorial matching for $\cC'$ by matching these extra pairs in $\M(x',c')$ and $\M(a',y')$, where there is a similar situation. This matching is maximal, provided $\mathcal{P}$ is maximal. If the signs of the points agree, we choose the ordering where the element from $\M(x,c)$, resp. $\M(a,y)$, is first.

\begin{remark}\label{rem:distortion}
Let $I_x$, resp.\ $I_y$, be an ordered pairing of the boundary points of $\Gamma_\mathcal{P}(x,d)$, resp.\ $\Gamma_\mathcal{P}(y,d)$. Then $I'=I_x\sqcup I_y$ is an ordered pairing of the boundary points of $\Gamma_{\mathcal{P}'}(x',d')$ and it is easy to see that
\[
 o(\Gamma_{\mathcal{P}'}(x',d'),I') = o(\Gamma_\mathcal{P}(x,d),I_x)+o(\Gamma_\mathcal{P}(y,d),I_y).
\]
However, when passing to $\mathcal{P}''$ we have fewer boundary points in $\Gamma_{\mathcal{P}''}(x',d')$ and extra circle components will occur. 

A similar problem will appear in $\Gamma_{\mathcal{P}''}(a',c')$ where also some circles can change their form. Assume for example that there are two positively signed points $A_1,A_2\in\M(a,x)$, and one point $B\in\M(x,c)$. Furthermore, there is one point $B'\in \M(y,c)$. A handle slide of $x$ over $y$ will create one extra point $B''\in \M(x',c')$ which we can match with $B$. The points $A_1$ and $A_2$ are matched in $\mathcal{P}$ and hence in $\mathcal{P}''$, but the edge between $(B,A_1)$ and $(B,A_2)$ coming from this matching is only present in $\Gamma_\mathcal{P}(a,c)$. In $\Gamma_{\mathcal{P}''}(a',c')$ we get edges between $(B,A_i)$ and $(B'',A_i)$ from the pairing $(B,B'')$. 
Note that the distortions $d_\mathcal{P}(a,c)$ and $d_\mathcal{P''}(a',c')$ are different, and in the algorithm below we need to add $1$ to $o(a',c')$.
\end{remark}

We say that $\cC'$ is obtained from $\cC$ by a \em double slide\em, if it is the result of two handle slides of an object $x$ over $y$ of the same sign. For a double slide we can do different ways of matching points, and in particular we have to worry less about new interval components. So the points of $\M(y,c)$, resp.\ $\M(a,x)$, which are matched in $\mathcal{P}$ now appear twice in $\M(x',c')$, resp.\ $\M(a',y')$, and we can keep this matching for $\mathcal{P}'$ for both copies. 
But if $B\in \M(y,c)$ is not matched in $\mathcal{P}$, we get two copies $B',B''\in \M(x',c')$ which we can match. Similarly, $A\in \M(a,x)$ leads to two points $A',A''\in \M(a',y')$ that we can match. As both signs are the same, the order matters, and we choose a consistent order $(A',A'')\in \mathcal{P}'$. However, for $B\in \M(y,c)$ we choose $(B',B'')\in \mathcal{P}'$ if $s(B)=0$, and $(B'',B')\in \mathcal{P}'$ if $s(B) = 1$.

\begin{lemma}\label{lem:doubleslide}
Let $\cC$ be a framed flow category and $\mathcal{P}$ a partial combinatorial matching. Let $\cC'$ be the result of double sliding an object $x$ over $y$. With the partial combinatorial matching $\mathcal{P}'$ for $\cC'$ described above, we get
\[
 \Gamma_{\mathcal{P}'}(x',d') = \Gamma_\mathcal{P}(x,d) \sqcup L_\mathcal{P}(y,d),
\]
where $|d|=|x|-2$, and $L_\mathcal{P}(y,d)$ is the special graph structure with empty vertex set, and one loop for every interval component in $\Gamma_\mathcal{P}(y,d)$. For $b\not=x$ with $|b|=|x|$ we get
\[
  \Gamma_{\mathcal{P}'}(b',d') = \Gamma_\mathcal{P}(b,d).
\]
For $|a|=|x|+1=|c|+2$ we get
\[
 \Gamma_{\mathcal{P}'}(a',c')=\Gamma_\mathcal{P}(a,c).
\]
\end{lemma}

\begin{proof}
We get
\[
 \M(x',d') = \M(x,d) \sqcup \M(y,d) \sqcup \M(y,d)
\]
and any non-trivially framed circle in $\M(y,d)$ appears twice and can be disregarded. Similarly, any circle component in $\Gamma_\mathcal{P}(y,d)$ appears twice in $\Gamma_{\mathcal{P}'}(x',d')$ and can be disregarded. If we have an interval with endpoints $(C,B)\in \M(c,d)\times \M(y,c)$ and $(\bar{C},\bar{B})\in \M(\bar{c},d)\times \M(y,\bar{c})$, we now get two intervals with endpoints $(C',B')$, $(\bar{C}',\bar{B}')$, and $(C',B'')$, $(\bar{C}',\bar{B}'')$. These two intervals are connected via directed edges from the partial matchings of $B'$, $B''$ and $\bar{B}',\bar{B}''$, and by the choice of matchings in $\mathcal{P}'$ we see that this component $C$ satisfies $o(C)=1$. We can therefore replace this component with a loop in $L_\mathcal{P}(y,d)$.

For $b\not=x$ with $|b|=|x|$ we get no change in the moduli spaces, so the result follows.

For $|a|=|x|+1=|c|+2$ each pair $(B,A)\in \M(y,c)\times \M(a,x)$ leads to two intervals with endpoints $(B',A)$, $(B,A')$ and $(B'',A)$, $(B,A'')$ in $\M(a',c')$. It is easy to see that these intervals combine to a component $C$ with $o(C)=0$, and hence can be disregarded in $\Gamma_{\mathcal{P}'}(a',c')$.
\end{proof}

If the cardinality of $\M(a,b)$ is odd, there is a point $A\in \M(a,b)$ which is not matched in any maximal partial combinatorial matching $\mathcal{P}$. We would like to assume that $s(A)=0$ if and only if $[a:b]>0$. This need not be preserved by handle slides, but we can adjust the matching.

\begin{example}
Assume $\M(x,c)$ contains four points, all positively framed, and matched accordingly by $\mathcal{P}$. Assume also that $\M(y,c)$ contains three points with identical framing, and after the handle slide, these three points have negatively framed copies in $\M(x',c')$. We want to change the partial combinatorial matching $\mathcal{P}'$ so that the non-matched point is positively framed. Essentially all we have to do is to break a match between two points $A_1,A_2$ and match one of them with the non-matched point $A\in \M(y,c)$.
\end{example}

For the next lemma, note that the boundary points of a special graph structure $\Gamma_\mathcal{P}(b,d)$ where $b,d$ are objects with $|b|=|d|+2$ and $\mathcal{P}$ is a maximal partial combinatorial matching are given by pairs $(B,A)\in \M(c,d)\times \M(b,c)$ with neither $B$ nor $A$ matched in $\mathcal{P}$. An ordered pairing $I(b,d)$ of boundary points is thus a collection of tuples $((B,A),(\bar{B},\bar{A}))$. If $s(B,A)=s(\bar{B},\bar{A})$ we interpret this order as a directed edge
\begin{center}
\begin{tikzpicture}
\node[circle,draw, inner sep=0pt, minimum size=3pt] at (0,0) {$ $};
\node[circle,draw, inner sep=0pt, minimum size=3pt] at (2,0) {$ $};
\node at (0,0.3) {$(B,A)$};
\node at (2,0.3) {$(\bar{B},\bar{A})$};
\draw[->] (0.05,0) -- (1,0);
\draw[-] (0.5,0) -- (1.95,0);
\end{tikzpicture}
\end{center}
while in the case $s(B,A)\not=s(\bar{B},\bar{A})$ we interpret this as an undirected edge
\begin{center}
\begin{tikzpicture}
\node[circle,draw, inner sep=0pt, minimum size=3pt] at (0,0) {$ $};
\node[circle,draw, inner sep=0pt, minimum size=3pt] at (2,0) {$ $};
\node at (0,0.3) {$(B,A)$};
\node at (2,0.3) {$(\bar{B},\bar{A})$};
\draw[-] (0.05,0) -- (1.95,0);
\end{tikzpicture}
\end{center}
By abuse of notation we identify such an edge with an element of $I(b,d)$.

Note that for a given point $A\in \M(b,c)$ not matched in $\mathcal{P}$ and a fixed object $d$ with $|d|=|c|-1$, there can be at most one boundary point $(B,A)$ in $\Gamma_\mathcal{P}(b,d)$.

\begin{lemma}\label{lem:changesign}
Let $\cC$ be a framed flow category and $\mathcal{P}$ a maximal partial combinatorial matching. Assume that $\M(b,c)$ is a $0$-dimensional moduli space with an odd number of points, and let $B$ be the point in $\M(b,c)$ not matched in $\mathcal{P}$. Assume that $s(B)$ is different from the sign of $[b:c]$. Then there is a maximal partial combinatorial matching $\mathcal{P}'$ such that the unmatched point $B'\in \M(b,c)$ in $\mathcal{P}'$ satisfies $s(B)\not=s(B')$, and the following holds.
\begin{enumerate}
\item Assume that $d$ is an object with $|d|=|b|-2$ such that $\Gamma_\mathcal{P}(b,d)$ has a boundary point $(C,B)$, that is, $\M(c,d)$ has an odd number of points.
Let $I(b,d)$ be a choice of ordered pairings of boundary points of $\Gamma_\mathcal{P}(b,d)$ with $((C,B),(\bar{C},\bar{B}))\in I(b,d)$ for some boundary point $(\bar{C},\bar{B})$. Let $I'(b,d)$ be the choice of ordered pairings of boundary points of $\Gamma_{\mathcal{P}'}(b,d)$ differing from $I(b,d)$ only as follows.

If 
\begin{tikzpicture}
\node[circle,draw, inner sep=0pt, minimum size=3pt] at (0,0) {$ $};
\node[circle,draw, inner sep=0pt, minimum size=3pt] at (2,0) {$ $};
\node at (0,0.3) {$(C,B)$};
\node at (2,0.3) {$(\bar{C},\bar{B})$};
\draw[->] (0.05,0) -- (1,0);
\draw[-] (0.5,0) -- (1.95,0);
\end{tikzpicture}
$\in I(b,d)$, then
\begin{tikzpicture}
\node[circle,draw, inner sep=0pt, minimum size=3pt] at (0,0) {$ $};
\node[circle,draw, inner sep=0pt, minimum size=3pt] at (2,0) {$ $};
\node at (0,0.3) {$(C,B')$};
\node at (2,0.3) {$(\bar{C},\bar{B})$};
\draw[-] (0.05,0) -- (1.95,0);
\end{tikzpicture}
$\in I'(b,d)$.

If 
\begin{tikzpicture}
\node[circle,draw, inner sep=0pt, minimum size=3pt] at (0,0) {$ $};
\node[circle,draw, inner sep=0pt, minimum size=3pt] at (2,0) {$ $};
\node at (0,0.3) {$(C,B)$};
\node at (2,0.3) {$(\bar{C},\bar{B})$};
\draw[-] (0.05,0) -- (1.95,0);
\end{tikzpicture}
$\in I(b,d)$, then
\begin{tikzpicture}
\node[circle,draw, inner sep=0pt, minimum size=3pt] at (0,0) {$ $};
\node[circle,draw, inner sep=0pt, minimum size=3pt] at (2,0) {$ $};
\node at (0,0.3) {$(C,B')$};
\node at (2,0.3) {$(\bar{C},\bar{B})$};
\draw[->] (0.05,0) -- (1,0);
\draw[-] (0.05,0) -- (1.95,0);
\end{tikzpicture}
$\in I'(b,d)$.

Furthermore,
\[
 o(\Gamma_\mathcal{P}(b,d),I(b,d)) = o(\Gamma_{\mathcal{P}'}(b,d),I'(b,d)).
\]
\item Assume that $a$ is an object with $|a|=|b|+1$ such that $\Gamma_\mathcal{P}(a,c)$ has a boundary point $(B,A)$, that is, $\M(a,b)$ has an odd number of points.
Let $I(a,c)$ be a choice of ordered pairings of boundary points of $\Gamma_\mathcal{P}(a,c)$ with $((B,A),(\bar{B},\bar{A}))\in I(a,c)$ for some boundary point $(\bar{B},\bar{A})$. Let $I'(a,c)$ be the choice of ordered pairings of boundary points of $\Gamma_{\mathcal{P}'}(a,c)$ differing from $I(a,c)$ only as follows.

If 
\begin{tikzpicture}
\node[circle,draw, inner sep=0pt, minimum size=3pt] at (0,0) {$ $};
\node[circle,draw, inner sep=0pt, minimum size=3pt] at (2,0) {$ $};
\node at (0,0.3) {$(B,A)$};
\node at (2,0.3) {$(\bar{B},\bar{A})$};
\draw[->] (0.05,0) -- (1,0);
\draw[-] (0.5,0) -- (1.95,0);
\end{tikzpicture}
$\in I(a,c)$, then
\begin{tikzpicture}
\node[circle,draw, inner sep=0pt, minimum size=3pt] at (0,0) {$ $};
\node[circle,draw, inner sep=0pt, minimum size=3pt] at (2,0) {$ $};
\node at (0,0.3) {$(B',A)$};
\node at (2,0.3) {$(\bar{B},\bar{A})$};
\draw[-] (0.05,0) -- (1.95,0);
\end{tikzpicture}
$\in I'(a,c)$.

If 
\begin{tikzpicture}
\node[circle,draw, inner sep=0pt, minimum size=3pt] at (0,0) {$ $};
\node[circle,draw, inner sep=0pt, minimum size=3pt] at (2,0) {$ $};
\node at (0,0.3) {$(B,A)$};
\node at (2,0.3) {$(\bar{B},\bar{A})$};
\draw[-] (0.05,0) -- (1.95,0);
\end{tikzpicture}
$\in I(a,c)$, then
\begin{tikzpicture}
\node[circle,draw, inner sep=0pt, minimum size=3pt] at (0,0) {$ $};
\node[circle,draw, inner sep=0pt, minimum size=3pt] at (2,0) {$ $};
\node at (0,0.3) {$(B',A)$};
\node at (2,0.3) {$(\bar{B},\bar{A})$};
\draw[->] (0.05,0) -- (1,0);
\draw[-] (0.05,0) -- (1.95,0);
\end{tikzpicture}
$\in I'(a,c)$.

Furthermore,
\[
 o(\Gamma_\mathcal{P}(a,c),I(a,c)) = o(\Gamma_{\mathcal{P}'}(a,c),I'(a,c)).
\]
\end{enumerate}
\end{lemma}

\begin{proof}
There is a matched pair $(B',B'')$ in $\mathcal{P}$ such that $s(B)\not= s(B')=s(B'')$, and we define $\mathcal{P}'$ by matching $(B,B'')$, leaving $B'$ unmatched, and keeping all other matchings from $\mathcal{P}$. Given a $C\in \M(c,d)$ we have intervals $I,I',I''$ with endpoints $(C,B),(C,B'),(C,B'')$. If $C$ is matched in $\mathcal{P}$, these intervals are matched in the same way in both $\Gamma_\mathcal{P}(b,d)$ and $\Gamma_{\mathcal{P}'}(b,d)$. If $C$ is not matched by $\mathcal{P}$, then $(C,B)$ is connected to another endpoint $(\bar{C},\bar{B})$ with either $s(C,B)=s(\bar{C},\bar{B})$ or not.

Up to move-equivalence, the situation is as follows in $\Gamma_\mathcal{P}(b,d)$ and $\Gamma_{\mathcal{P}'}(b,d)$.
\begin{center}
\begin{tikzpicture}
\node[circle,draw, inner sep=0pt, minimum size=3pt] at (0,0) {$ $};
\node[circle,draw, inner sep=0pt, minimum size=3pt] at (1,0) {$ $};
\node[circle,draw, inner sep=0pt, minimum size=3pt] at (0,1) {$ $};
\node[circle,draw, inner sep=0pt, minimum size=3pt] at (1,1) {$ $};
\node[circle,draw, inner sep=0pt, minimum size=3pt] at (2,1) {$ $};
\node[circle,draw, inner sep=0pt, minimum size=3pt] at (3,1) {$ $};
\node[circle,draw, inner sep=0pt, minimum size=3pt] at (0,2) {$ $};
\node[circle,draw, inner sep=0pt, minimum size=3pt] at (1,2) {$ $};
\node[circle,draw, inner sep=0pt, minimum size=3pt] at (2,2) {$ $};
\node[circle,draw, inner sep=0pt, minimum size=3pt] at (3,2) {$ $};
\node at (0,-0.4) {$(C,B'')$};
\node at (0,1.3) {$(C,B')$};
\node at (0,2.3) {$(C,B)$};
\node at (3,2.3) {$(\bar{C},\bar{B})$};
\node at (1.3,2) {$($};
\node at (1.6,2) {$)$};
\node at (1,0.7) {$\varepsilon$};
\node at (1,1.7) {$\varepsilon+1$};
\node at (1.5,2.8) {$\Gamma_\mathcal{P}(b,d)$};
\draw[-] (0.05,0) -- (0.95,0);
\draw[-,blue] (1.05,0) -- (1.45,0);
\draw[-,blue] (3.05,1) -- (3.45,1);
\draw[-] (0.05,1) -- (0.95,1);
\draw[-,blue] (1.05,1) -- (1.95,1);
\draw[-] (2.05,1) -- (2.95,1);
\draw[-] (0.05,2) -- (0.95,2);
\draw[->,blue] (1.05,2) -- (1.5,2);
\draw[-,blue] (1.5,2) -- (1.95,2);
\draw[-] (2.05,2) -- (2.95,2);
\draw[->,blue] (0,0.95) -- (0,0.5);
\draw[-,blue] (0,0.5) -- (0,0.05);

\node[circle,draw, inner sep=0pt, minimum size=3pt] at (5,0) {$ $};
\node[circle,draw, inner sep=0pt, minimum size=3pt] at (6,0) {$ $};
\node[circle,draw, inner sep=0pt, minimum size=3pt] at (5,1) {$ $};
\node[circle,draw, inner sep=0pt, minimum size=3pt] at (6,1) {$ $};
\node[circle,draw, inner sep=0pt, minimum size=3pt] at (7,1) {$ $};
\node[circle,draw, inner sep=0pt, minimum size=3pt] at (8,1) {$ $};
\node[circle,draw, inner sep=0pt, minimum size=3pt] at (5,2) {$ $};
\node[circle,draw, inner sep=0pt, minimum size=3pt] at (6,2) {$ $};
\node[circle,draw, inner sep=0pt, minimum size=3pt] at (7,2) {$ $};
\node[circle,draw, inner sep=0pt, minimum size=3pt] at (8,2) {$ $};
\node at (5,-0.4) {$(C,B'')$};
\node at (5,1.3) {$(C,B')$};
\node at (5,2.3) {$(C,B)$};
\node at (8,2.3) {$(\bar{C},\bar{B})$};
\node at (6.3,2) {$($};
\node at (6.6,2) {$)$};
\node at (6,0.7) {$\varepsilon$};
\node at (6,1.7) {$\varepsilon+1$};
\node at (6.5,2.8) {$\Gamma_{\mathcal{P}'}(b,d)$};
\draw[-] (5.05,0) -- (5.95,0);
\draw[-,blue] (6.05,0) -- (6.45,0);
\draw[-,blue] (8.05,1) -- (8.45,1);
\draw[-] (5.05,1) -- (5.95,1);
\draw[-,blue] (6.05,1) -- (6.95,1);
\draw[-] (7.05,1) -- (7.95,1);
\draw[-] (5.05,2) -- (5.95,2);
\draw[->,blue] (6.05,2) -- (6.5,2);
\draw[-,blue] (6.5,2) -- (6.95,2);
\draw[-] (7.05,2) -- (7.95,2);
\draw[blue] plot [smooth] coordinates {(5,1.95) (4.7,1.6) (4.5,1) (4.7,0.4) (5,0.05)};
\end{tikzpicture}
\end{center}
where there may or may not be a directed edge in the top row. Using a move (a) or (i), we see that the graph structure $\Gamma_{\mathcal{P}'}(b,d)$ is equivalent to
\begin{center}
\begin{tikzpicture}
\node[circle,draw, inner sep=0pt, minimum size=3pt] at (5,0) {$ $};
\node[circle,draw, inner sep=0pt, minimum size=3pt] at (6,0) {$ $};
\node[circle,draw, inner sep=0pt, minimum size=3pt] at (5,1) {$ $};
\node[circle,draw, inner sep=0pt, minimum size=3pt] at (6,1) {$ $};
\node[circle,draw, inner sep=0pt, minimum size=3pt] at (7,1) {$ $};
\node[circle,draw, inner sep=0pt, minimum size=3pt] at (8,1) {$ $};
\node[circle,draw, inner sep=0pt, minimum size=3pt] at (5,2) {$ $};
\node[circle,draw, inner sep=0pt, minimum size=3pt] at (6,2) {$ $};
\node[circle,draw, inner sep=0pt, minimum size=3pt] at (7,2) {$ $};
\node[circle,draw, inner sep=0pt, minimum size=3pt] at (8,2) {$ $};
\node at (5,-0.4) {$(C,B'')$};
\node at (5,1.3) {$(C,B')$};
\node at (5,2.3) {$(C,B)$};
\node at (8,2.3) {$(\bar{C},\bar{B})$};
\node at (6.3,1.5) {$($};
\node at (6.6,1.5) {$)$};
\node at (6,0.7) {$\varepsilon$};
\node at (6,2.3) {$\varepsilon+1$};
\draw[-] (5.05,0) -- (5.95,0);
\draw[-,blue] (6.05,0) -- (6.45,0);
\draw[-,blue] (8.05,1) -- (8.45,1);
\draw[-] (5.05,1) -- (5.95,1);
\draw[-,blue] (6.025,1.975) -- (6.4,1.6);
\draw[->,blue] (6.975,1.025) -- (6.75,1.25);
\draw[-,blue] (6.8,1.2) -- (6.6,1.4);
\draw[-] (7.05,1) -- (7.95,1);
\draw[-] (5.05,2) -- (5.95,2);
\draw[->,blue] (6.025,1.025) -- (6.5,1.5);
\draw[-,blue] (6.5,1.5) -- (6.975,1.975);
\draw[-] (7.05,2) -- (7.95,2);
\draw[blue] plot [smooth] coordinates {(5,1.95) (4.7,1.6) (4.5,1) (4.7,0.4) (5,0.05)};
\end{tikzpicture}
\end{center}
Note that the diagonal edge in the component with $(C,B')$ is directed if and only if the original edge in the component of $(C,B)$ is not directed. Up to equivalence, the difference between $\Gamma_\mathcal{P}(b,d)$ and $\Gamma_{\mathcal{P}'}(b,d)$ is therefore as described.

The case of $\Gamma(a,c)$ is essentially the same and will be omitted.
\end{proof}

\begin{remark}
Note that the possibility \begin{tikzpicture}
\node[circle,draw, inner sep=0pt, minimum size=3pt] at (0,0) {$ $};
\node[circle,draw, inner sep=0pt, minimum size=3pt] at (2,0) {$ $};
\node at (0,0.3) {$(\bar{C},\bar{B})$};
\node at (2,0.3) {$(C,B)$};
\draw[->] (0.05,0) -- (1,0);
\draw[-] (0.05,0) -- (1.95,0);
\end{tikzpicture}
$\in I(x,d)$ is ruled out by requiring $((C,B),(\bar{C},\bar{B}))\in I(x,d)$. If we would require $((\bar{C},\bar{B}),(C,B))\in I(x,d)$ we could get an analogous situation with the arrows reversed.

It is important to note that the order $((C,B),(\bar{C},\bar{B}))$ is arbitrary, but needs to be consistently applied to each object $d$.
\end{remark}

\subsection{The algorithm}\label{subsec:algorithm}

We may turn a general framed $1$-flow category $\cC$ into a framed $1$-flow category in primary Smith normal form as follows. We begin with objects of maximal degree, and move on to lower degrees. More formally, assume our $1$-flow category satisfies the conditions
\begin{enumerate}
 \item[($1_m$)] If $b$ is an object with $|b|\geq m$, there is at most one non-empty $0$-dimensional moduli space $\M(a,b)$, which contains $p^k$ elements, all of which have the same framing, where $p\geq 2$ is a prime and $k\geq 1$.
 \item[($2_m$)] For each object $a$ with $|a|>m$ there exists at most one object $b$ with $|a|=|b|+1$ and $\M(a,b)$ non-empty.
\end{enumerate}

If $m\geq|a|$ for all objects $a\in \Ob(\cC)$, these conditions are trivially satisfied, and if $m \leq |a|$ for all objects $a\in \Ob(\cC)$, these conditions are equivalent to $\cC$ being in primary Smith normal form.

So let us assume that a framed $1$-flow category $\cC$ satisfies ($1_m$) and ($2_m$), and let us turn it into a framed $1$-flow category $\cC'$ satisfying ($1_{m-1}$) and ($2_{m-1}$).

Let $a$ be an object with $|a|=m$. If there exists an object $u$ with $|u|=m+1$ and $\M(u,a)$ non-empty, then the sum of the signs of all elements in a $0$-dimensional moduli space $\M(a,b)$ has to be zero as $\partial^2=0$. Using Whitney tricks, we can make all $0$-dimensional moduli spaces $\M(a,b)$ empty.

So assume $\M(u,a)$ is empty for all $u$ with $|u|=m+1$. Use Whitney tricks so that the number of points in $\M(a,b)$ is the same as the modulus of the sum of the signs of the elements in it.

Among the non-empty moduli spaces, find the object $b$ with $|b|=m-1$ such that the cardinality $|\M(a,b)|$ where we vary over $b$ is minimal. Then perform handle slides of objects $a'$ over $a$ (and Whitney tricks) to get a lower cardinality $|\M(a',b')|>0$ or make all $\M(a',b)$ empty. We can repeat this until we get objects $a,b$ with $\M(a',b)=\emptyset$ for all objects $a'\not=a$ with $|a'|=|a|$. We then slide objects $b$ over $b'$ to reduce the cardinality $|\M(a,b')|>0$ or make them empty. Again, after finitely many steps (and possibly sliding $a'$ over $a$ again), we get objects $a,b$ with $\M(a,b)$ has finitely many points, $\M(a,b')=\emptyset=\M(a',b)$ for all objects $a'\not=a$, $b'\not=b$ with $|a'|=|a|$ and $|b'|=|b|$.

Repeating this process will eventually get the $1$-flow category into Smith normal form, but ignoring the primary part. But this step can be done at the end with appropriate handle slides. 

During the handle slides, we need to keep track of the $1$-dimensional moduli spaces, or rather the special graph structures $\Gamma_\mathcal{P}(a,c)$ with $|a|=|c|+2$. In fact, we only need to keep track of the number of circle components $C$ with $o(C)=1$ and the interval components. Also, each interval component is equivalent to one that consists of only one standardly framed edge if the endpoints have different sign, or it consists of three edges if the endpoints have the same sign, one which is ordered and the other two standardly framed.

So for each pair $a,c$ of objects with $|a|=|c|+2$ let $I(a,c)$ be a collection of ordered pairs of the boundary points of $\Gamma_\mathcal{P}(a,c)$. Then define
\[
o(a,c):=o(\Gamma_\mathcal{P}(a,c),I(a,c))
\]

\begin{remark}\label{rem:newvalue}
When performing a handle slide, the boundary points of $\Gamma_\mathcal{P'}(a',c')$ are obtained from the old boundary points in $I(a,c)$ and the new edges described in Lemma \ref{lem:singleslide} after gluing in extra edges coming from the new pairings in $\mathcal{P'}$. This can result in new circle components $C$ for which we can calculate $o(C)$. This value needs to contribute to $o(a',c')$. Similarly, there can be long intervals which can be shortened to one or three (in case of a directed edge) edges by possibly adding loops. Such loops also need to contribute to $o(a',c')$.
\end{remark}

The algorithm can be described as follows.

{\bf Step 1} Obtain a framed $1$-flow category, for example using \cite{LipSarSq} or \cite{JLS}.

{\bf Step 2} Choose a maximal partial combinatorial matching $\mathcal{P}$, which means one groups the elements of $0$-dimensional moduli spaces $\M(a,b)$ into pairs, possibly leaving one point unpaired if there is an odd number of points. This can be done so that the unpaired point has the same sign as $[a:b]$.

{\bf Step 3} Form $\Gamma_\mathcal{P}(a,c)$ for every two objects $a,c$ with $|a|=|c|+2$. The endpoints of intervals are given by points $(B,A)\in \M(b,c)\times \M(a,b)$ for some $b$ where both $A$ and $B$ are not paired in $\mathcal{P}$. Minimize the length of intervals in the equivalence class of $\M(a,b)$ and determine $o(a,c)$ and $I(a,c)$.

{\bf Step 4} Perform handle slides according to the Smith normalization process, starting with objects of maximal degree. If we have a single slide of $x$ over $y$, we obtain
\[
 o(a',y')=o(a,y)+o(a,x)
\]
for objects $a$ with $|a|=|x|+2$. For objects $b,d$ with $|b|=|x|+1=|d|+2$ the values $o(b',d')$ are obtained from $o(b,d)$ and the modifications described in Remark \ref{rem:newvalue}. We also need to consider the distortion described in Remark \ref{rem:distortion}, thus adding 1 to $o(b',d')$ for every pair of different points $A_1,A_2\in \M(b,x)$ together with each $B\in \M(y,d)$.

For objects $e$ with $|e|=|x|-2$ the values $o(y',e')$ are obtained from $o(x,e)$ and $o(y,e)$ and taking into account possible new loops coming from the modifications on the intervals, compare Remark \ref{rem:distortion}.

In the case of a double slide we only need to change $o(x',e')$ according to Lemma~\ref{lem:doubleslide}.

{\bf Step 5} After each handle slide, ensure that the non-matched points in $\M(a,b)$ have the same sign as $[a:b]$, using Lemma \ref{lem:changesign}.

Repeat Step 4 and 5 until the $1$-flow category is in primary Smith normal form.

\bibliographystyle{amsalpha}
\def\MR#1{}
\bibliography{algorithm}
\end{document}